\documentclass[11pt,reqno]{amsart}
\usepackage{amsmath,amssymb,amsthm}
\usepackage{hyperref}
\usepackage[colorinlistoftodos]{todonotes}
\presetkeys{todonotes}{inline, color=green}{}

\usepackage{dsfont}
\usepackage{mathrsfs}

\numberwithin{equation}{section}

\usepackage{url}
\usepackage{textcase}
\usepackage{mathtools}
\usepackage{stmaryrd} 
\usepackage[inline]{enumitem}
\usepackage{graphicx,setspace}
\usepackage{tikz-cd} 
\usetikzlibrary{arrows}
\usepackage{wasysym}
\usepackage[makeroom]{cancel}
\usepackage{fullpage}
\usepackage{comment}
\usepackage{enumitem}
\usepackage{caption}
\allowdisplaybreaks
\newcommand{\bb}{\mathbb}
\newtheorem{theorem}{Theorem}[section]
\newtheorem*{theorem*}{Theorem}
\newtheorem{lemma}[theorem]{Lemma}
\newtheorem{claim}[theorem]{Claim}
\newtheorem{proposition}[theorem]{Proposition}

\newtheorem{corollary}[theorem]{Corollary}

\newtheorem{conjecture}{Conjecture}
\newtheorem{definition}[theorem]{Definition}

\theoremstyle{definition}{

\newtheorem*{definition*}{Definition}

\newtheorem*{question*}{Question}
\newtheorem*{example*}{Example}
\newtheorem*{examples*}{Examples}

\newtheorem*{remark*}{Remark}

}

\newcommand\abs[1]{\left|#1\right|}
\newcommand\pth[1]{\left ( #1 \right )}
\newcommand\brak[1]{\left [ #1 \right ]}

\newcommand{\norm}[1]{\left\lvert\left\lvert#1\right\rvert\right\rvert}
\newcommand{\floor}[1]{\left\lfloor#1\right\rfloor}
\newcommand{\ceil}[1]{\left\lceil#1\right\rceil}
\newcommand{\ip}[1]{\left < #1 \right >}

\newcommand{\fk}{\mathfrak}
\newcommand{\mc}{\mathcal}
\newcommand{\ms}[1]{\mathscr{#1}}

\newcommand\nc{\newcommand}
\nc{\wt}{\widetilde}
\nc{\kernel}{\text{ker}}
\nc{\image}{\text{Im}}

\renewcommand{\d}[1]{\mathrm{d}#1}
\renewcommand{\bf}[1]{\textbf{#1}}
\renewcommand{\rm}[1]{\mathrm{#1}}

\hypersetup{
    colorlinks=true, 
    linkcolor=blue,  
    urlcolor=magenta, 
}

\def\ep{\varepsilon}

\def\p{\partial}

\def\N{\mathbb{N}}
\def\Z{\mathbb{Z}}

\def\R{\mathbb{R}}
\def\C{\mathbb{C}}

\def\P{\mathbb{P}}

\def\O{\mathcal{O}}
\def\E{\mathbb{E}}

\def\GOE{\mathrm{GOE}}
\def\GUE{\mathrm{GUE}}
\def\AS{\mathrm{AS}}
\def\ld{\lambda}
\def\a{\mathrm{a}}
\def\Pf{\mathrm{Pf}}
\def\fc{\mathfrak{c}}

\def\Ai{\mathrm{Ai}}

\begin{document}

\title[The Lower Tail of the Half-Space KPZ Equation]{The Lower Tail of the Half-Space KPZ Equation}
\author{Yujin H. Kim}
\address{Y.\ H.\ Kim\hfill\break
Courant Institute\\ New York University\\
251 Mercer Street\\ New York, NY 10012, USA.}
\email{\textcolor{blue}
{\href{mailto:yujin.kim@cims.nyu.edu}{yujin.kim@cims.nyu.edu}}}

\subjclass[2010]{Primary: 60H15. Secondary: 60B20,  45M05, 60F10, 60G55, 60H25.}

\keywords{(Half-space) Kardar-Parisi-Zhang equation, Pfaffian point processes, GOE ensemble, large deviations, stochastic Airy operator, Ablowitz-Segur Solution to Painlev\'{e} II}

\begin{abstract}
We establish the first tight bound on the lower tail probability of the half-space KPZ equation with Neumann boundary parameter $A = -1/2$ and narrow-wedge initial data.
When the tail depth is of order $T^{2/3}$,
the lower bound demonstrates a crossover between a regime of super-exponential decay with exponent $\frac{5}{2}$ (and leading pre-factor $\frac{2}{15 \pi}T^{1/3}$) and a regime with exponent $3$ (and leading pre-factor $\frac{1}{24}$); the upper bound demonstrates a crossover between a regime with exponent $\frac{3}{2}$ (and arbitrarily small pre-factor) and a regime with exponent $3$ (and leading pre-factor $\frac{1}{24}$). 
We show that, given a crude leading-order asymptotic in the \textit{Stokes region} (Definition~\ref{def:stokes},
first defined in (Duke Math J., \cite{Bot17})) for the Ablowitz-Segur solution to the Painlev\'e II equation, the upper bound on the lower tail probability can be improved to demonstrate the same crossover as the lower bound.
We also establish novel bounds on the large deviations of the GOE point process.


\end{abstract}

\maketitle

\tableofcontents

\section{Introduction}
\label{sec:intro}

The Kardar-Parisi-Zhang (KPZ) equation is formally given by 
\begin{align}
\p_T H(T,X) = \frac{1}{2} \p_X^2 H(T,X) + \frac{1}{2} \pth{\p_X H(T,X)}^2 + \xi(T,X) \,,
\label{KPZ_defn}
\end{align}
where $T\geq 0$, $X \in \R$, and  $\xi$ is Gaussian space-time white noise with covariance  $\E \brak{\xi(T,X) \xi(S,Y)} = \delta(T-S) \delta(X-Y)$.
A  physically relevant notion of solution to this equation is given by the
\emph{Cole-Hopf solution to the KPZ equation} with \emph{narrow-wedge initial data}
\begin{align}
    H(T,X) :=  \log Z(T,X), 
    \ \ \ \ \ 
    \text{with } Z(0,X) = \delta_0(X)\,,
    \label{Cole-Hopf_soln}
\end{align}
where $Z$ solves the $(1+1)\rm{d}$ stochastic heat equation (SHE) with multiplicative space-time white noise
\begin{align}
    \p_T Z(T,X) = \frac{1}{2}\p_X^2 Z(T,X) + Z(T,X) \xi(T,X).
    \label{SHE_defn}
\end{align}
The well-definedness of \eqref{Cole-Hopf_soln} is given by the work of
\cite{Mue91} establishing almost-sure positivity of $Z$ for a wide class of initial data (including delta initial data).

The KPZ equation is a paradigmatic model in a class of models, known as the KPZ universality class,
whose long-time limit is the KPZ fixed point. While this universality class is not strictly defined, all models in this class should share specific salient features.
The KPZ equation itself has been shown to govern the long-time limits under weak asymmetric scaling of many other models in the universality class. The notes and surveys \cite{Fer10}, \cite{Cor12}, \cite{Cor14} \cite{QS15}, \cite{Sas16}, \cite{Tak18}, and \cite{Zyg18} provide further reading on various aspects of the KPZ universality class.

Just as in the full-space case, the \textit{half-space KPZ equation with Neumann boundary conditions} plays a significant role within the half-space KPZ universality class. 
Mathematical analysis of the half-space analogues of  models believed to lie in the KPZ universality class began with the work of \cite{BR01, IS04}, both of which consider variants of half-space TASEP.
For a recent result relating to half-space TASEP, see \cite{BBCS16}.
Progress has been especially fruitful in the case of ASEP.
\cite{CS16} established convergence of the height function of half-space ASEP under weakly asymmetric scaling to the half-space KPZ equation with Neumann boundary parameter $A \geq 0$. 
Following this result, \cite{BBCW17} established an exact one-point distribution formula for half-space ASEP with $A = - 1/2$, and \cite{Par19} was able to extend the work of \cite{CS16} to show convergence to the half-space KPZ equation for all real $A$.
See, for instance, \cite{KLD18a}, \cite{Wu18}, \cite{KLD19}, \cite{BBC20}, \cite{BKL20}, \cite{Lin20}, and  \cite{BFO21} for additional results in the half-space KPZ universality class. 

We now describe the half-space KPZ equation in detail.

\subsection{The half-space KPZ equation with Neumann boundary conditions}
\label{subsec:intro-hskpz}
This paper seeks to establish bounds on the lower tail of the half-space KPZ equation with Neumann boundary condition, an object which we presently define.

\begin{definition}[Mild solution to the half-space SHE, half-space KPZ]
We say $\ms{Z}(T,X)$ is a \bf{mild solution} to the SHE given in \eqref{SHE_defn} on $\R_+$ with delta initial data at the origin and \textbf{Robin boundary condition} with parameter $A \in \R$
\begin{align}
    \p_X \ms{Z}(T,X) \bigg |_{X=0} &= A \ms{Z}(T,0) \,, ~\forall T >0 \,,
\end{align}
if $\ms{Z}(T, \cdot)$ is adapted to the filtration given by 
$\sigma\pth{ \ms{Z}(0, \cdot), W|_{[0,T]} }$ 
and
the following Duhamel-form identity is satisfied
\begin{align}
    \ms{Z}(T,X) &= \int_0^{\infty} \ms{P}_T^R(X,Y) Z(0,Y) ~dY \\
    &+ 
    \int_0^T \int_0^{\infty} \ms{P}^R_{T-S}(X,Y) Z(S,Y) \xi(S,Y) ~dW_S(dY) 
\end{align}
for all $T>0$ and $X >0$. Here, the last integral is It\^o with respect to the cylindrical Wiener process $W$, and $\ms{P}^R$ is the heat kernel on $[0,\infty)$, i.e., the fundamental solution to the heat equation on $[0, \infty)$, satisfying the Robin boundary condition
\begin{align}
    \p_X \ms{P}_T^R(X,Y) \bigg |_{X=0} &= A \ms{P}_T^R(0,Y) \,, ~\forall T> 0 \,, ~Y> 0 \,.
\end{align}
The Hopf-Cole solution to the \textbf{half-space KPZ equation with Neumann boundary parameter}~$A$ is then defined to be $H = \log \ms{Z}$.
\end{definition}

\cite[Proposition~4.2]{Par19} establishes  the existence, uniqueness, and almost-sure positivity of $\ms{Z}(T, \cdot)$ for all $A \in \R$, which makes the Hopf-Cole solution to the half-space KPZ equation with Neumann boundary condition $A \in \R$ well-defined.

Our paper establishes tight bounds on the lower tail probability of $H(T,0)$, that is, the probability that $\ms{Z}(T,0)$ is very close to $0$, or equivalently, that $H(T,0)$ is very negative, for the critical boundary parameter $A = -1/2$.
Our result builds on the method used by \cite{CG18} to find analogous bounds for the full-space KPZ lower tail.

We now explain the choice of boundary parameter $A =  - 1/2$. For this particular boundary parameter, \cite[Theorem~1.1]{Par19} established Tracy-Widom GOE fluctuations at the origin.

\begin{proposition}[{\cite[Theorem~1.3]{Par19}}]
\label{Par19.thm.1.1}
Let $H(T,X)$ be the solution to the half-space KPZ equation with inhomogeneous Neumann boundary parameter $A = -1/2$ and narrow-wedge initial data (which corresponds to $\delta_0$ initial data for the SHE). Then the following weak convergence result holds
\begin{align}
    \lim_{T \to \infty} \P(\Upsilon_T \leq s)
    = F_{\GOE}(s) \,, \ \ \ \ 
    \rm{where}  \ \ 
    \Upsilon_T &:= 
    \dfrac{H(2T,0) + \frac{T}{12}}{T^{1/3}}\,. \label{Intro.Y_T}
\end{align}
Here, $F_{\GOE}(s)$ is the Tracy-Widom GOE fluctuations \cite{TW96},
and $\Upsilon_T$ is the solution to the KPZ equation after centering and re-scaling.
\end{proposition}

For other choices of $A$, establishing the limiting fluctuations of $\Upsilon_T$ has been elusive, 
and thus establishing lower tail bounds in these regimes seems at the moment unfeasible. 
\cite[Conjecture~1.2]{Par19}
gives a conjecture establishing exactly two more regimes of distinct fluctuations: $A < - 1/2$, with Gaussian fluctuations, and $A>-1/2$, with Tracy-Widom GSE distribution \cite{TW96}. 
\cite[Section~1.3]{Par19} gives a heuristic argument for the Gaussianity of the $A<-1/2$ regime; see also \cite{Par19-2}. \cite{GLD12, BBC16, KLD19} provides strong evidence towards the conjectured $A>-1/2$ regime, though we emphasize that no part of this conjecture has been rigorously established.

On the other hand, for $A= -1/2$, we have
access to Proposition \ref{Intro.ltf}, which provides the starting point for our analysis.

\begin{proposition}[\cite{Par19}]
\label{Intro.ltf}
Let $H(T,X)$ denote the solution to the half-space KPZ equation on $[0,\infty)$ with Neumann boundary parameter $A = -1/2$ and narrow-wedge initial data. Then for $u > 0$, 
\begin{align}
    \E_{\rm{SHE}} \brak{ \exp \pth{ - u \exp \pth{H(2T,0) + \frac{T}{12}}
    }} 
    &=
    \E_{\GOE} \brak{
    \prod_{k=1}^{\infty} \dfrac{1}{\sqrt{
    1+ 4 u \exp \pth{T^{1/3} \a_k}
    }}
    }. \label{Intro.ltf.eq}
\end{align}
Here, the $(\a_1 > \a_2 > \dots)$ form the GOE point process (defined in Section \ref{subsection_GOEPP}).
\end{proposition}

Taking     $u := \frac{1}{4} \exp \pth{T^{1/3} s}$ in \eqref{Intro.ltf.eq}
and recalling $\Upsilon_T$ from \eqref{Intro.Y_T}, we obtain
\begin{align}
    \E_{\rm{SHE}} \brak{\exp \pth{- \frac{1}{4}\exp \pth{T^{1/3}(\Upsilon_T + s)}}}  
    &= 
    \E_{\GOE} \brak{
    \prod_{k=1}^{\infty} \dfrac{1}{\sqrt{
    1+\exp \pth{ T^{1/3} \pth{\a_k + s}}
    }}
    } 
    \label{Intro.lpf_usub}.
\end{align}
Note that the function 
$\exp \pth{ -\exp (x) }$ is an  approximate of the indicator function $\mathds{1}(x \leq 0 )$, and so the integrand of the left-hand side of \eqref{Intro.lpf_usub} approximates $\P(\Upsilon_T + s \leq 0)$ for large $s$. 
This heuristic is made rigorous in Section \ref{section.proof.thm.1.1.anal}.
Proposition \ref{Intro.ltf} was conjectured in \cite[Theorem~7.6]{BBCW17}, which proves the analogous formula for the height function of half-space ASEP and computes asymptotics which were expected to lead to the above result on the KPZ equation.
Combining their result with \cite[Theorem~1.2]{Par19}  yields Proposition \ref{Intro.ltf}.

We our now ready to state our main result, Theorem~\ref{thm.1.1.anal}, which establishes upper and lower bounds on the lower       tail probability $\P(\Upsilon_T \leq - s)$ for large but fixed times $T>0$.
\begin{theorem}
\label{thm.1.1.anal}
Let $\Upsilon_T$ denote the solution to the half-space KPZ equation with Neumann boundary parameter $A = -1/2$ and narrow-wedge initial data, centered and re-scaled as in~\eqref{Intro.Y_T}.
Fix any $\eta >0$,
$\ep \in (0,1/3)$, $\delta \in (0,1/4)$, and $T_0 >0$. There exist positive constants $S := S (\eta, \ep, \delta, T_0)$, $C := C(T_0)$, $K_1 := K_1(\ep, \delta, T_0)$, and $K_2 := K_2(T_0)$ 
such that for all $s \geq S$ and $T \geq T_0$, we have 
\begin{align}
    \P \pth{\Upsilon_T \leq -s}
    &\geq 
    e^{
    -\frac{2(1+C\ep)}{15\pi}T^{1/3} s^{5/2}
    }
    + 
    e^{- K_2 s^3} \,, \label{1.5.eqn.anal}
\end{align}
and 
\begin{align}
    \P \pth{\Upsilon_T \leq -s}  
    &\leq 
    e^{
    -\frac{2(1- C\ep)}{15 \pi} T^{1/3} s^{5/2}
    }
    +
    e^{ -\frac{\ep}{2} sT^{1/3} -\eta s^{3/2}}
    +e^{-\frac{1-C\ep}{24}s^3} \,.
    \label{1.4.eqn.anal.weak}
\end{align}
Assuming
Conjecture~\ref{conjecture}, we have the stronger
\begin{align}
    \P \pth{\Upsilon_T \leq -s}  
    &\leq 
    e^{
    -\frac{2(1- C\ep)}{15 \pi} T^{1/3} s^{5/2}
    }
    +
    e^{ -\frac{\ep}{2} sT^{1/3} -K_1 s^{3- \delta}}
    +e^{-\frac{1-C\ep}{24}s^3} \,.
    \label{1.4.eqn.anal}
\end{align}
\end{theorem}
Conjecture~\ref{conjecture} has a rather technical statement regarding the leading-order asymptotics of \textit{Ablowitz-Segur solution} $u_{\AS}(x; \gamma)$ to the \textit{Painlev\'e II equation} in a certain region, named the \textit{Stokes region}. Its openness is due to the difficulty of a certain Riemann-Hilbert problem. 
One major goal of this article is to highlight the direct connection between leading-order asymptotics of $u_{\AS}(x; \gamma)$ in the Stokes region and the lower-tail of the KPZ equation, in hopes of motivating further study of the Stokes region. For the sake of a more stream-lined discussion of Theorem~\ref{thm.1.1.anal} and its proof, we postpone a detailed discussion of Conjecture~\ref{conjecture} and the Painlev\'e II equation to Section~\ref{subsec:uas}.
The proof of Theorem~\ref{thm.1.1.anal} is given in Section \ref{section.proof.thm.1.1.anal}. We note that \eqref{1.4.eqn.anal.weak} and \eqref{1.4.eqn.anal} differ only in the second term of each.

We can see Theorem~\ref{thm.1.1.anal}  displays three distinct regions of decay as follows. 
First, note that Proposition \ref{Par19.thm.1.1} implies that,  as $T \to \infty$, $\P(\Upsilon_T < - s)$ should decay according to $F_{\GOE}(-s)$, which is approximately $\exp\pth{-\frac{1}{24}s^3}$ for large $s$ (see Proposition \ref{prop.TW.GOE.dist}). 
This  cubic decay is exhibited in the last terms of \eqref{1.5.eqn.anal}---\eqref{1.4.eqn.anal}. 
Note that in the range $T^{2/3} \gg s \gg 0$,
either the second or the third term of \eqref{1.4.eqn.anal} dominates; in \eqref{1.5.eqn.anal}, the second term dominates
(though in the lower bound \eqref{1.5.eqn.anal}, the prefactor of the cubic exponent is not explicit). 
When $T \to \infty$, the third term of \eqref{1.4.eqn.anal} dominates and thus recovers the cubic decay of the $F_{\GOE}$ tail.
On the other hand, in the ``short time deep tail" region $s \gg T^{2/3}$, the first term of both \eqref{1.4.eqn.anal} and \eqref{1.5.eqn.anal} dominates; however, in \eqref{1.4.eqn.anal.weak}, the second term dominates the first term in all regions. 
The $5/2$ exponent and the $\frac{2}{15 \pi}$ prefactor for this region were first observed in \cite{KLD18}.
The crossover from $5/2$ to cubic exponent that occurs  when $s$ is of order $T^{2/3}$
can be understood in terms of large deviations: as $T \to \infty$, the crossover is exhibited by the large deviation rate function for the half-space KPZ equation, which has speed $T^2$. 
In the full-space case, this crossover was first predicted by
\cite{SMP17}, which also contains the first prediction of the full-space rate function; \cite{CGK+18, KLD18a, KLDP18} each provide alternative methods of computing this rate function. In particular, \cite{CGK+18} showed that the half-space rate function is simply one-half that of the full space.
The rate functions for both the full and half-space case were finally rigorously established by \cite{Tsa18}.
Just over a year after the posting of this paper, the preprint \cite{Zho20} obtained sharper upper and lower bounds than in Theorem~\ref{thm.1.1.anal}
by proving large deviation bounds for the Airy point process. 
In particular, their upper-bound on the lower tail probability is given by
$e^{
    -\frac{2(1- C\ep)}{15 \pi} T^{1/3} s^{5/2}
    }
    +
    e^{ -\frac{\ep}{2} sT^{1/3} -Ks^3}
    +e^{-\frac{1-C\ep}{24}s^3}$,
so that the aforementioned crossover from exponent $5/2$ to $3$ is attained.
Large deviation bounds for the Airy point process were originally (non-rigorously) derived by \cite{CGK+18} using Coulomb gas heuristics.

The techniques used to prove Theorem~\ref{thm.1.1.anal} are heavily inspired by the work of \cite{CG18} on the lower tail of the full-space KPZ equation. Their work starts with
the full-space KPZ analog to \eqref{Intro.ltf.eq}, which was established in \cite{BG16}, where the full-space KPZ equation is related to a multiplicative functional of the Airy (GUE) point process by manipulations of an exact formula for the one-point distribution of SHE with delta initial data. This one-point distribution formula was simultaneously and independently computed in 
\cite{ACQ11, SS10, CLDR, Dot10} and rigorously proved in \cite{ACQ11}. 
In \cite{CG18}, the formula of \cite{BG16} was manipulated to yield tight bounds on the lower tail of the full-space KPZ equation; however, in order to do this, \cite{CG18} first establishes appropriate control on the fluctuations of the GUE point process.
Their work strongly suggests that a careful manipulation of \eqref{Intro.lpf_usub} would similarly yield tight bounds on the lower tail of the half-space KPZ equation, given analogous control on the GOE point process; indeed, this is the approach taken in the current article. We now outline our approach to studying the GOE point process and the methods used therein.

\subsection{Fluctuations of the GOE point process}
\label{subsec:intro-GOEPP}
In Section~\ref{subsection_GOEPP}, we define the GOE point process and describe its key properties as a \textit{simple Pfaffian point process} (also defined in that section). 
The estimates on the GOE point process needed in this article pertain to \textbf{(1)} controlling the locations of individual GOE points, and \textbf{(2)} controlling the number of GOE points within intervals. 

Towards \textbf{(1)}, we detail in Section~\ref{subsection.SAO} the well-studied connection between the \textit{(stochastic) Airy operator} (SAO) and the GOE points, and describe the relevant known results (Propositions~\ref{rrv11_thm}---\ref{airy_eval_bound}). 
In particular, the seminal work of \cite{RRV11} (Proposition~\ref{rrv11_thm}) gives an equivalence in distribution between the eigenvalues of the $\beta =1$ SAO and the GOE points, while \cite[Proposition~4.5]{CG18} (Proposition~\ref{prop4.5_cg18} below) establishes uniform control on the deviations of the (random) SAO eigenvalues from deterministic locations given by the eigenvalues $(\ld_k)$ of the (deterministic) \textit{Airy operator}. Theorem~\ref{1.6_anal} below is then simply the combination of Proposition~\ref{rrv11_thm} and Proposition~\ref{prop4.5_cg18}.
\begin{theorem}
\label{1.6_anal}
For $\ep \in (0,1)$, let $C_{\ep}^{\GOE}$ be the smallest real number such that, for all $k \geq 1$, 
\begin{align}
    (1- \ep) \lambda_k - C_{\ep}^{\GOE} \leq 
    -\a_k
    \leq (1+ \ep) \lambda_k + C_{\ep}^{\GOE} \,, \label{1.6.anal.eqn}
\end{align}
where $\a_k$ is the $k^{th}$ largest point of the GOE point process and $\lambda_k$ is the $k^{th}$ smallest eigenvalue of the Airy operator.
Then, for all $\ep, \delta \in (0,1)$, there exist constants $S_0 := S_0(\ep, \delta)$ and $\kappa := \kappa( \ep, \delta)$ such that, for all $s \geq S_0$, 
\begin{align}
    \P(C_{\ep}^{\GOE} \geq s) \leq \kappa \exp \pth{- \kappa s^{1-\delta}}. 
\end{align}
\end{theorem}
Theorem~\ref{1.6_anal} establishes an upper bound on the probability that the $\a_k$ deviate away from the (deterministic) $\ld_k$, uniformly in $k$.
This is extremely helpful because we know what the $\ld_k$ look like: Proposition \ref{airy_eval_bound} tells us 
that\footnote{Here, $f(k) \sim g(k)$ if they are asymptotically equivalent, i.e., $\lim_{k \to \infty} \frac{f(k)}{g(k)} = 1$.} 
$\ld_k \sim \pth{\frac{3\pi}{2}k}^{2/3}$.

Towards \textbf{(2)}, we define the counting function
\[
\chi^{\GOE} : \mc{B}(\R) \to \Z_{\geq 0}, 
\ \ \ \ \ 
\chi^{\GOE}(B) := \#\{k: \a_k \in B\}, ~\forall B \in \mc{B}(\R),
\]
where $\mc{B}(\R)$ denotes the Borel $\sigma$-algebra of $\R$.  $\chi^{\GOE}(\cdot)$ is a non-negative integer-valued random measure on $(\R, \mc{B}(\R), \mu)$, where $\mu$ denotes the Lebesgue measure on $\R$, that, informally speaking, counts the number of GOE points in a Borel set $B$--- see Section~\ref{subsection_GOEPP} for a formal description. We will also refer to $\chi^{\GOE}$ as the GOE point process. The mean of $\chi^{\GOE}$ on intervals is given by Theorem~\ref{1.3_anal_prop} below, which is proved at the end of Section~\ref{subsection_GOEPP}.

\begin{theorem}
\label{1.3_anal_prop}
Define the interval $\fk{B}_1(s) := [-s, \infty)$.
For any $s > 0$, we have
\begin{align}
\E_{\GOE}\brak{\chi^{\GOE}(\fk{B}_1(s))}
&= \frac{2}{3\pi}s^{3/2} + D_1(s) \label{1.3_exp},
\end{align}
where $\sup_{s > 0} \abs{D_1(s)} < \infty$.
\end{theorem}

We expect that this result and other statistics for $\chi^{\GOE}$ should be known; however, we were unable to find such results in the literature.
Note that the leading-order term $s^{3/2}$ of \eqref{1.3_exp} matches the leading-order term of the expectation of the GUE (or, Airy) point process $\chi^{\rm{Ai}}$ on $\fk{B}_1(s)$, computed in \cite{Sos00}. 
\cite{Sos00} also computes the variance of and establishes a central limit theorem for $\chi^{\rm{Ai}}$.

In light of Theorem~\ref{1.3_anal_prop}, we are interested in deviations of order $s^{3/2}$ of $\chi^{\GOE}$ on intervals of size $s$. 
The upper deviations result (Theorem~\ref{1.5_anal}, proved in Section~\ref{section.thm.1.5.}) will actually follow from the results discussed in \textbf{(1)} and the lower deviations result (Theorem~\ref{1.4.final}, proved in Section~\ref{section.thm.1.4.proof}), and so we now turn our attention to the lower deviations.
To introduce important related objects and motivate the results that follow, we begin with a preliminary computation of the lower deviations of $\chi^{\GOE}$.
Recall from Theorem~\ref{1.3_anal_prop} the interval $\fk{B}_1(s)$. For any $s\in \R$ and $v>0$, define
\[
    F_1(s,v):= \E \brak{ 
    \exp \pth{
    -v \chi^{\GOE}\pth{
    \fk{B}_1(s)
    }
    }
    }\,.
\]
$F_1(s,v)$ is the \textit{cumulant generating function} for $\chi^{\GOE}$.
Now, for any positive $c, v$ and $s$, taking $f(x) = e^{-v x}$ in Markov's inequality and then applying Theorem~\ref{1.3_anal_prop} yields
\begin{align}
    &\P \pth{\chi^{\GOE}(\fk{B}_1(s)) - \E[ \chi^{\GOE}(\fk{B}_1(s))] \leq -cs^{3/2}}
    \nonumber \\
    &\leq \exp \pth{- c v s^{3/2} + v \E \brak{\chi^{\GOE}(\fk{B}_s)}}
    F_1(-s,v)\,,
    \nonumber \\
    &= \exp \pth{ 
    \Big( \frac{2}{3\pi}-c \Big) v s^{3/2}  
    + v D_1(s)} F_1(-s, v)\,,
    \label{eqn:lowerdev-markov-1}
\end{align}
Thus, we see that in order to achieve decay in \eqref{eqn:lowerdev-markov-1} for any $c>0$, one needs to achieve an upper-bound like\footnote{Here, we use ``little-Oh" notation: $f(s)$ is called  $o(1)$ if $\lim_{s\to \infty} f(s) = 0$.} 
\begin{align}
    F_1(-s,v) \leq \exp\pth{ - \frac{2}{3\pi} vs^{3/2}(1+ o(1))}\,,
    \label{eqn:F1-goal}
\end{align}
for some choice of $v$.
Obtaining \eqref{eqn:F1-goal} for optimal $v$ will be a major technical focus of this article. An important step towards this end is Theorem~\ref{1.7.rewrite.F_1.corollary} below. Before giving this result, we must first uncover a connection to the \emph{thinned GOE/GUE point processes}
with parameter $\gamma := \gamma(v) = 1-e^{-v}$ and  the \textit{Ablowitz-Segur solution} to the \textit{Painlev\'e II equation} (this connection is developed further in Section~\ref{section.thm.1.7}).

The Ablowitz-Segur (AS) solution $u_{\rm{AS}}(\cdot ,\gamma)$ to the Painlev\'e II equation is a one parameter family of solutions to 
\[
u''_{\rm{AS}} = x u_{\rm{AS}} + 2u^3_{\rm{AS}}
\]
with the boundary condition
\begin{align}
u_{\AS}(x; \gamma) = 
\sqrt{\gamma} \frac{x^{-1/4}}{2\sqrt{\pi}} e^{-\frac{2}{3}x^{3/2}} \pth{1+ o(1)}
\label{eqn:uas-bdry}
\end{align}
as $x \to \infty$. When $\gamma =1$, $u_{\rm{AS}}$ is called the Hastings-McLeod solution and typically denoted $u_{\rm{HM}}$. This particular solution was introduced in \cite{HM80}, where they solved the connection problem, that is, gave an asymptotic formula for $u_{\rm{HM}}(x)$ as $x \to -\infty$. 
For $\gamma  \in (0,1)$ fixed, the connection problem for $u_{\rm{AS}}$ was partially solved by \cite{AS77a, AS77b}.

The thinned version of a point process with parameter $\gamma$ removes each particle independently with probability $1-\gamma$; we discuss the thinned GOE point process formally in Section~\ref{subsection.thinned.goe.painleve}.
In Theorem~\ref{1.7.rewrite.F_1}, we prove by way of a 
\textit{Fredholm Pfaffian} 
formula (defined in Section \ref{subsection.fredholm.pfaffian}) 
that 
\[
F_1(s,v) = \mc{F}_1(s,v)\,, \text{ for all $s \in \R$ and $v\geq 0$}\,,
\]
where $\mc{F}_1(s,v)$ denotes the distribution function of the largest particle of the thinned GOE point process with parameter $\gamma(v)$.
Let $\mc{F}_2(s,v)$ denote the distribution function of the largest particle of the thinned GUE point process with parameter $\gamma(v)$. 
In Proposition \ref{1.7.prop.bb17}, we recall a formula from \cite{BB17} that relates $\mc{F}_1(s,v)$ to $\mc{F}_2(s,2v)$ and $u_{\rm{AS}}$, described in the next subsection.
It is a result of \cite{CG18}, restated here as Proposition~\ref{prop.thm.1.7.cg.18}, that 
\[
\mc{F}_2(s,v) = F_2(s,v) := \E \brak{ 
    \exp \pth{
    -v \chi^{\rm{Ai}}\pth{
    [s, \infty)
    }
    }
    }, \text{ for all $s \in \R$ and $v \geq 0$.}
\]
Combining Proposition~\ref{1.7.prop.bb17}, Proposition~\ref{prop.thm.1.7.cg.18}, and Theorem~\ref{1.7.rewrite.F_1} yields Theorem~\ref{1.7.rewrite.F_1.corollary}, which yields a formula for $F_1(s,v)$ in terms of $F_2(s,v)$ and $u_{\AS}$. Theorem~\ref{1.7.rewrite.F_1.corollary} is proved in Section~\ref{section.1.7.final.proof}.
\begin{theorem}
\label{1.7.rewrite.F_1.corollary}
Fix any $s \in \R$ and $v \geq 0$. Define $\gamma  := \gamma(v) =  1- e^{-v}$ and 
$\gamma_2 :=  \gamma_2(v) = 1- e^{-2v}$; note that $\gamma_2 \in [0,1)$. Then 
\begin{align}
    F_1(s,v) 
    &= \sqrt{F_2(s, 2v)} 
    \sqrt{1 + \dfrac{\cosh \mu(s, \gamma_2) - \sqrt{\gamma_2} \sinh \mu(s, \gamma_2) -1 }{2- \gamma}}
    \label{1.7_F_1}
\end{align}
where 
\[
\mu(s, \gamma_2):= \int_{s}^{\infty} u_{\AS}(x; \gamma_2) ~dx.
\]
\end{theorem}
In Corollary~\ref{cor:bb17-fixed}, we give an asymptotic expansion for $F_1(s,v)$ for any \textit{fixed} $v >0$ that satisfies \eqref{eqn:F1-goal}, thus yielding exponential decay on the right-hand of \eqref{eqn:lowerdev-markov-1} with exponent $-s^{3/2}$. This yields equation~\eqref{eqn:1.4.weak} of Theorem~\ref{1.4.final}.
However, the authors of \cite{CG18} found optimum decay of $F_2(s,2v)$ when $v =  \frac{1}{2}s^{\frac{3}{2}-\delta}$. Indeed,  part of \cite[Theorem~1.7]{CG18} (recorded here as Proposition~\ref{prop.thm.1.7.cg.18}) states that, for any $\delta \in (0,2/5)$, as $s \to \infty$,
\begin{align}
    F_2(-s, 2 \bar{v}) \leq \exp \Big( - \frac{2}{3\pi} s^{3-\delta} + \mc{O}(s^{3- \frac{13\delta}{11}}) \Big) \,.
    \label{eqn:cg18-F2bd}
\end{align}

Fix $\delta \in (0,2/5)$. Throughout this paper, we fix 
\begin{align}
    \bar{v} := \bar{v}(s, \delta) = \frac{1}{2}s^{\frac{3}{2}-\delta} \quad \text{ and} \quad 
    \bar{\gamma} := \gamma_2(\bar{v}) = 1- \exp(-s^{\frac{3}{2}-\delta}) \,. 
    \label{def:bar-v-gamma}
\end{align}
Now, take $v := \bar{v}$ in the notation of Theorem~\ref{1.7.rewrite.F_1.corollary}. Then upon substituting \eqref{eqn:cg18-F2bd} into Theorem~\ref{1.7.rewrite.F_1.corollary}, we see that obtaining the bound
\begin{align}
\exp (\abs{\mu(-s, \bar{\gamma})} = \exp(o(s^{3-\delta}))
    \label{eqn:desire-mu-AS}
\end{align}
would actually yield \eqref{eqn:F1-goal} with $v = \bar{v}$ there. The result would be exponential decay on the right-hand side  \eqref{eqn:lowerdev-markov-1} with exponent $-s^{3-\delta}$ instead of $-s^{3/2}$. 
Thus, \textit{showing \eqref{eqn:desire-mu-AS} translates directly into a vastly improved bound on the right-hand of \eqref{eqn:lowerdev-markov-1}.}

To achieve \eqref{eqn:desire-mu-AS}, one  needs to control  $u_{\AS}(x;\bar{\gamma})$ for all $x \in [-s, \infty)$ and $s \to \infty$.
While much is known about both $u_{\AS}(x; \gamma)$ and $\mu(s; \gamma)$ for values of $\gamma$ fixed (with respect to $x$), much less is understood for general values of $\gamma$.
As we show in the following subsection, 
there is a particular region of $x$, known as the \textit{Stokes region}, on which  leading-order asymptotics of $u_{\AS}(x; \bar{\gamma})$ do not exist at this time. This lack of knowledge prevents us from bounding in absolute value the integral of $u_{\AS}(x; \bar{\gamma})$ on the Stokes region, and therefore, we can not establish \eqref{eqn:desire-mu-AS}; however, we  show that given crude leading-order asymptotics on $u_{\AS}$ in the Stokes region (see Conjecture~\ref{conjecture}), we can obtain \eqref{eqn:desire-mu-AS}.

\subsection{Asymptotics of the Ablowitz-Segur solution to the Painlev\'e II equation}
\label{subsec:uas}
In this subsection, we recall what is known and unknown about the  asymptotic properties  of the Ablowitz-Segur solution to the Painlev\'e II equation  as both $x$ and $\gamma$ vary and detail what these results imply for $F_1(s,v)$.

As explained in the last paragraph of the previous subsection, we are interested in $u_{\AS}(x; \bar{\gamma})$ over $x \in [-s,\infty)$, where  $\bar{\gamma} := 1- \exp(-s^{\frac{3}{2}-\delta})$, for any $\delta \in (0,2/5)$.
Our goal is to show \eqref{eqn:desire-mu-AS}, for which we seek appropriate leading-order asymptotics of $u_{\AS}(x;\bar{\gamma})$ as $x \to -\infty$. To understand  $u_{\AS}(x;\gamma)$ for $\gamma$ that may vary with $x$, we turn to the important work of Bothner \cite{Bot17}, which contains the most up-to-date results on such  asymptotics in the case $x \to -\infty$ and $\abs{\gamma} \uparrow 1$ (\textit{regular transition} in \cite{Bot17}) or the case $x \to -\infty$ and $\abs{\gamma} \downarrow 1$ (\textit{singular transition} in \cite{Bot17}).
These results were achieved via a non-linear steepest descent analysis applied to a certain Riemann-Hilbert problem. 
Since $s \to \infty$,  we are interested in the regular transition results of \cite{Bot17}.
To state these results, we define the following parameter for any $x \in \R$ and $\gamma \in [0,1)$:
\begin{align}
    \aleph := \aleph(x, \gamma) = \frac{-1}{(-x)^{3/2}}\log(1-\gamma) \,.
    \label{def:aleph}
\end{align}
Note that the  exponential decay in \eqref{eqn:uas-bdry} implies that for any constant $x_0 >0$, the integral of $u_{\AS}(x; \bar{\gamma})$ over $[-x_0 ,\infty)$ is bounded.
The remaining region $x \in [-s, -x_0)$ is contained in  $\aleph \in (0, \infty)$.
For any  $\zeta \in (0, \frac{2\sqrt{2}}{3})$, Theorems~$1.10$ and~$1.12$ of \cite{Bot17} achieve asymptotic expressions for $u_{\AS}(x;\gamma)$ as $x \to -\infty$ in the regions 
$\aleph \in I_1(\zeta) := \big(0, \frac{2\sqrt{2}}{3}- \zeta \big]$
and 
$\aleph \in I_2 := \big[ \frac{2\sqrt{2}}{3} , \infty \big)$, 
respectively.\footnote{Actually, the expression holds for any fixed $f \in \R$ and $I_2(f) := \big[ \frac{2\sqrt{2}}{3} - \frac{f}{(-x)^{3/2}}, \infty \big)$. However, considering $f$ large (but fixed) does not change our results asymptotically, and so we simply take $f= 0$.}
\cite[Theorem~1.12]{Bot17} is transcribed here as Proposition~\ref{prop:Bot17-thm1.12}. \cite[Theorem~1.10]{Bot17} gives an expression in terms of Jacobi theta functions and elliptic integrals that is pseudoperiodic. In Lemma~\ref{lem:mu:I21}, we manipulate this result to show that there exists
$\zeta_0 \in (0, \frac{2\sqrt{2}}{3})$ such that 
$u_{\AS}(x ; \bar{\gamma}) = \mc{O}((-x)^{1/2})$  uniformly over $\aleph \in \big(0, \frac{2\sqrt{2}}{3}- \zeta_0 \big]$ as $x \to -\infty$. From Lemma~\ref{lem:mu:I21} and Proposition~\ref{prop:Bot17-thm1.12},   it follows almost immediately that
\begin{align}
\int_{\aleph \in I_1(\zeta_0) \cup I_2} \abs{u_{\AS}(x; \bar{\gamma})} ~dx = \mc{O}(s^{3/2})\,.
\label{eqn:intro-nonstokes}
\end{align}
In \cite{Bot17}, $I_1(\zeta)$ is named the \textit{regular Boutroux region} and $I_2$ the \textit{Hastings-McLeod region}; the remaining region of $\aleph>0$ was named the \textit{Stokes region}.
\begin{definition}[Stokes region]
\label{def:stokes}
For any $\zeta \in (0,\frac{2\sqrt{2}}{3})$, the region $\aleph \in ( \frac{2\sqrt{2}}{3}- \zeta, \frac{2\sqrt{2}}{3})$ is referred to as the \textbf{Stokes region}. 
\end{definition}
\cite{Bot17} does not give a full asymptotic expression for $u_{\AS}(x; \gamma)$ in the Stokes region, stating that ``the nonlinear steepest descent analysis becomes increasingly difficult."
Moreover, at the time of this paper's release, it appears that no progress has been made towards such results in the Stokes region \cite{Bot21}. As a result, not enough is currently known about $u_{\AS}$ in the Stokes region to obtain \eqref{eqn:desire-mu-AS}, and thus we can not at present achieve \eqref{eqn:F1-goal} with $v = \frac{1}{2}s^{\frac{3}{2}-\delta}$ for any $\delta \in (0,2/5)$.

However, \textit{observe that only a crude upper-bound on $u_{\AS}(x; \bar{\gamma})$ is needed in order to show} \eqref{eqn:desire-mu-AS}. Indeed,  for $\bar{\aleph} := \aleph(x, \bar{\gamma})$, the part of the Stokes region that we are interested in is $( \frac{2\sqrt{2}}{3} - \zeta_0, \frac{2\sqrt{2}}{3})$, which is equivalent to 
\begin{align}
x \in \mathbf{I}_0 := \mathbf{I}_0(s,\delta) = \big( \! \!- \!(\tfrac{2\sqrt{2}}{3}- \zeta_0 )^{-2/3} s^{1- \frac{2}{3}\delta}, -( \tfrac{2\sqrt{2}}{3})^{-2/3} s^{1-\frac{2}{3}\delta} \big )\,.
\label{eqn:stokes-x-interval}
\end{align}
Note that $\mathbf{I}_0$ has length $C s^{1- \frac{2}{3}\delta}$, where $C$ denotes some constant.
\begin{conjecture}\label{conjecture}
Fix $\delta \in (0,2/5)$. Recall $\bar{\gamma} := \bar{\gamma}(s,\delta)$ from~\eqref{def:bar-v-gamma}, and recall $\mathbf{I}_0 := \mathbf{I}_0(s,\delta)$ from~\eqref{eqn:stokes-x-interval}. As $s\to\infty$, we have the following uniformly over all $x \in \mathbf{I}_0$ (equivalently, $\bar{\aleph} := \aleph(x, \bar{\gamma}) \in ( \frac{2\sqrt{2}}{3} - \zeta_0, \frac{2\sqrt{2}}{3}))$:
\begin{align}
    \abs{u_{\AS}(x; \bar{\gamma})} = o(s^{2 - \frac{\delta}{3}}) \,.
    \label{eqn:conjecture}
\end{align}
\end{conjecture}
Assuming Conjecture~\ref{conjecture}, we immediately have 
\begin{align}
    \int_{\mathbf{I}_0} \abs{u_{\AS}(x; \bar{\gamma})} ~dx  = o(s^{3-\delta})\,,
\end{align}
so that~\eqref{eqn:desire-mu-AS} follows from~\eqref{eqn:intro-nonstokes} and the last display. To be precise, we have the following results.

\begin{lemma}\label{1.7.mu_est}
Fix $\delta \in (0,2/5)$. Recall the function $\mu$ from Theorem~\ref{1.7.rewrite.F_1.corollary}.
There exist  positive constants $\mc{C} := \mc{C}(\delta)$ and $S_0 := S_0(\delta)$ such that for all $s\geq S_0$, 
\begin{align}
    \abs{\mu(-s, \bar{\gamma})} \leq  \mc{C} s^{3/2} + \abs{ \int_{\mathbf{I}_0} u_{\AS}(x;\bar{\gamma})~dx }\,.
    \label{1.7.mu_est.eqn.weak}
\end{align}
Assuming Conjecture~\ref{conjecture},
we  have the following expression as $s \to \infty$,
\begin{align}
    \abs{\mu(-s, \bar{\gamma})} = o(s^{3-\delta}) \,.
    \label{1.7.mu_est.eqn}
\end{align}
\end{lemma}

Lemma~\ref{1.7.mu_est} is proved  in Section~\ref{proof_of_mu_est}.
Combining this result with  Theorem~\ref{1.7.rewrite.F_1.corollary} and \eqref{eqn:cg18-F2bd} will yield the following bound.
\begin{theorem} \label{1.7.final}
Assume Conjecture~\ref{conjecture}.
For $\delta \in (0,2/5)$, we have the following expression as $s \to \infty$
\begin{align}
    F_1 \pth{ -s, \frac{1}{2}s^{\frac{3}{2}-\delta} }
    &\leq 
     \exp \pth{
    -\frac{1}{3 \pi}s^{3-\delta}(1+o(1))
    } \,. \label{1.7.final.eqn}
\end{align}
\end{theorem}
Theorem~\ref{1.7.final} is proved in Section~\ref{section.1.7.final.proof}.

Regarding evidence for the validity of Conjecture~\ref{conjecture},
we note that a leading-order expression for~$u_{\AS}(x; \bar{\gamma})$ was obtained in \cite[Theorem~1.13]{Bot17} for the portion of the Stokes region satisfying
\[ \aleph \geq \tfrac{2\sqrt{2}}{3} - f_3 \tfrac{\log((-x)^{3/2}}{(-x)^{3/2}} \,, \]
for any $f_3 < 7/6$. The expression shows that $u_{\AS}(x; \bar{\gamma}) = \mc O(x^{1/2})$ uniformly over this region of $\aleph$, which is consistent with 
Conjecture~\ref{conjecture}.
We note further that the bound  in~\eqref{eqn:conjecture} is much looser than both the aforementioned result and the existing leading-order asymptotics given for $u_{\AS}(x; \bar{\gamma})$ outside of the Stokes region (Proposition~\ref{prop:Bot17-thm1.12} and Lemma~\ref{lem:mu:I21}). Beyond these observations, we do not attempt to provide further justification for Conjecture~\ref{conjecture}.

\subsection{Main results on the GOE Point Process}
\label{subsec:mainresults}
Theorems \ref{1.4.final} and \ref{1.5_anal} establish the first bounds on the fluctuations of $\chi^{\GOE}$ below and above its mean, respectively, and may be of independent interest.

\begin{theorem}\label{1.4.final}
Fix any $\eta >0$, $c > 0$, and $\delta \in (0,2/5)$. There exists a positive constant $S_0 := S_0(\eta,c)$ such that for all $s \geq S_0$, 
\begin{align}
    \P \pth{\chi^{\GOE}[-s,\infty) - \E[ \chi^{\GOE}([-s,\infty))] \leq -cs^{3/2}} \leq 
    \exp \pth{ -\eta s^{3/2} }\,.
    \label{eqn:1.4.weak}
\end{align}
Furthermore, assuming Conjecture~\ref{conjecture}, there exist positive constants $S_0 := S_0(\delta)$ and $K := K(\delta)$ such that for all $s \geq S_0$ and $c > 0$,
\begin{align}
    \P \pth{\chi^{\GOE}(\fk{B}_1(s)) - \E[ \chi^{\GOE}(\fk{B}_1(s))] \leq -cs^{3/2}} &\leq 
    \exp{
    \pth{
    -\frac{1}{2}cs^{3-\delta}(1+ o(1)) 
    }
    } \, , \label{1.4.final.eq}
\end{align}
where  $\fk{B}_1(s) := [-s, \infty)$.
\end{theorem}
Theorem~\ref{1.4.final} is proved in Section~\ref{section.thm.1.4.proof}, essentially by combining \eqref{eqn:lowerdev-markov-1}, \eqref{1.7.final.eqn}, and~\eqref{eqn:bb17-fixed}.

\begin{theorem}
\label{1.5_anal}
Consider the intervals
\begin{align}
    \fk{B}_1(\ell) &:= [-\ell, \infty) \text{, and } \nonumber \\
    \fk{B}_k(\ell) &:= [-k \ell, -(k-1)\ell) ~\text{for } k \in \Z_{>1} \, . \nonumber
\end{align}
Fix $c> 0$ and $\delta \in (0,2/5)$. There exist 
$L_0 := L_0(c, \delta)$
and 
$\mc{C} := \mc{C}(c, \delta) >0$
such that, for all $\ell \geq L_0$ and for all $k \in \Z_{\geq 1}$, we have 
\begin{align}
    \P\pth{
    \chi^{\GOE}( \fk{B}_k(\ell)) - \E \brak{ \chi^{\GOE}(\fk{B}_k(\ell)) } \geq c\ell^{3/2}
    } 
    \leq \exp \big ( -\mc{C} \ell^{1-\delta} 
    \big ) \, . \label{1.5.anal.thm.eqn}
\end{align}
\end{theorem}
Theorem~\ref{1.5_anal} is proved in Section~\ref{section.thm.1.5.}.

\subsection{Outline}
\label{subsec:outline}
We now give an outline for the remainder of the article.
    In Section~\ref{sec:mainthm-proof}, we prove Theorem~\ref{thm.1.1.anal} by  realizing the left-hand side of the Laplace transform formula \eqref{Intro.lpf_usub} as an approximate indicator function for $\P(\Upsilon_T < - s)$.
    This translates our problem into bounding a multiplicative functional of the GOE point process, i.e., the right-hand side of \eqref{Intro.lpf_usub}. These bounds are given by Proposition \ref{prop.4.2.anal}.

    We next turn to a fine analysis of the GOE point process, which involves estimating the typical locations of the GOE points in large intervals and bounding their deviations from these locations. In Section~\ref{section.goe}, we define the GOE point process (and \textit{Pfaffian} point processes in general), and use known results on its correlation functions to prove Theorem~\ref{1.3_anal_prop}. 
    We then
    discuss the important connection with the eigenvalues of the \textit{stochastic Airy operator} (abbreviated SAO). In particular, the result of \cite{RRV11} (Proposition~\ref{rrv11_thm} below) gives an equivalence in distribution between  the GOE points and the negatives of the SAO eigenvalues. Furthermore,  \cite[Proposition 4.5]{CG18} (Proposition \ref{prop4.5_cg18} below) establishes  an upper bound on deviations of the SAO eigenvalues (uniformly over all eigenvalues) from their ``typical locations", which are given by the  eigenvalues of the \textit{Airy operator}. The locations of these deterministic eigenvalues are given by a result of \cite{MT59} (Proposition \ref{airy_eval_bound} below). Combining  Proposition \ref{rrv11_thm} and Proposition \ref{prop4.5_cg18} yields Theorem~\ref{1.6_anal}. Thus, we are able to effectively estimate the locations of individual GOE points.

    In Section~\ref{section.thm.1.7}, we turn our attention to the cumulant generating function $F_1(-s,v)$ for the GOE point process. The importance of this function was established in equation~\ref{eqn:lowerdev-markov-1} of Section~\ref{subsec:intro-GOEPP}.
    Via a Fredholm Pfaffian formula for $F_1(-s,v)$, we prove in Theorem~\ref{1.7.rewrite.F_1} a key equality between $F_1(-s,v)$ and the distribution function of the largest eigenvalue of the \textit{thinned GOE point process}. This allows us to translate the work of \cite{BB17} on this distribution function to $F_1(-s,v)$, which in particular leads to the  proofs of Theorem \ref{1.7.rewrite.F_1.corollary} and  (assuming Lemma~\ref{1.7.mu_est}) Theorem~\ref{1.7.final} in Section~\ref{section.1.7.final.proof}.
    Lemma~\ref{1.7.mu_est} is proved in Section~\ref{proof_of_mu_est}.

    In Sections~\ref{section.thm.1.4.proof} and~\ref{section.thm.1.5.}, we prove Theorems~\ref{1.4.final} and~\ref{1.5_anal} respectively. 
    Theorem~\ref{1.4.final} is proved essentially by substituting the results of Corollary~\ref{cor:bb17-fixed} and Theorem~\ref{1.7.final} into \eqref{eqn:lowerdev-markov-1}.
    Our strategy for proving Theorem~\ref{1.5_anal} involves approximating the number of GOE points in a closed interval of length $s$ by carefully estimating the nearest GOE points to the endpoints of this interval, and then bounding the fluctuations of these GOE points via Theorem~\ref{1.6_anal}.

    In Section~\ref{section.proof.upper.and.lower}, we apply our work on the GOE point process to prove Proposition~\ref{prop.4.2.anal}.

\subsection*{Acknowledgements}
We are grateful to Ivan Corwin for suggesting this problem and for providing helpful comments on numerous drafts of this paper, to Promit Ghosal for providing  guidance and insight at several  stages of this project, and to Thomas Bothner for enlightening discussions about the current literature on leading-order asymptotics for the Ablowitz-Segur solution to the Painlev\'e II equation in various regimes.
We are also grateful to  Guillaume Barraquand, 
Pierre le Doussal, Alexandre Krajenbrink, Yier Lin,
Baruch Meerson,
Li-Cheng Tsai, and
 Shalin Parekh for discussions and conversations related to this work.
 Finally, we are grateful to the anonymous referees for their time and effort in providing many important comments and suggestions.
 The author was partially funded by Ivan Corwin's Packard Fellowship for Science and Engineering while working on this paper. The author is currently supported by the National Science Foundation Graduate Research Fellowship under Grant No.~1839302.

\section{Proof of the main theorem}
\label{sec:mainthm-proof}
We begin by establishing  
upper and lower bounds on the right-hand side of the Laplace transform formula \eqref{Intro.lpf_usub}
in Proposition \ref{prop.3.1.anal}.

\begin{proposition}
\label{prop.3.1.anal}
Fix any $\eta >0$, $\ep \in (0,1/3)$, $\delta \in (0,1/4)$, and $T_0 >0$. There exist positive constants $S_0 := S_0 (\eta, \ep, \delta, T_0)$,  $C:=C(T_0)$, $K_1 := K_1(\ep, \delta, T_0) >0$, and $K_2 := K_2(T_0) > 0$ 
such that for all $s \geq S_0$ and $T \geq T_0$, we have 
\begin{align}
    \E \brak{\exp \pth{- \frac{1}{4}\exp \pth{T^{1/3}(\Upsilon_T + s)}}}   
    &\geq 
    e^{
    -\frac{2(1+C\ep)}{15\pi}T^{1/3} s^{5/2}
    }
    + 
    e^{- K_2 s^3} 
    \label{3.2.eqn.anal}
\end{align}
and
\begin{align}
    \E \brak{\exp \pth{- \frac{1}{4}\exp \pth{T^{1/3}(\Upsilon_T + s)}}}   
    &\leq 
    e^{
    -\frac{2(1- C\ep)}{15 \pi} T^{1/3} s^{5/2}
    }
    +
    e^{ -\frac{\ep}{2} sT^{1/3} -\eta s^{3/2}}
    +e^{-\frac{1-C\ep}{24}s^3}\,.
    \label{3.1.eqn.anal.weak}
\end{align}
Assuming Conjecture~\ref{conjecture}, we have the stronger upper bound
\begin{align}
    \E \brak{\exp \pth{- \frac{1}{4}\exp \pth{T^{1/3}(\Upsilon_T + s)}}}  
    \leq 
    e^{
    -\frac{2(1- C\ep)}{15 \pi} T^{1/3} s^{5/2}
    }
    +
    e^{ -\frac{\ep}{2} sT^{1/3} -K_1 s^{3- \delta}}
    +e^{-\frac{1-C\ep}{24}s^3} \,.
    \label{3.1.eqn.anal}
\end{align}
\end{proposition}

We prove Proposition \ref{prop.3.1.anal}
 in Section \ref{section.proof.prop.3.1.anal}.
We now prove the main result.

\subsection{Proof of Theorem~\ref{thm.1.1.anal}} \label{section.proof.thm.1.1.anal}
From Markov's inequality, we have
\begin{align*}
    \P( \Upsilon_T \leq - s ) 
    &= 
    \P \pth{ 
    \exp \pth{
        -\frac{1}{4} \exp \pth{
        T^{1/3}(\Upsilon_T + s)
        }
        }
    \geq e^{-1/4} \nonumber
    } \\
    &\leq 
    e^{1/4} \E \brak{ 
    \exp \pth{
        -\frac{1}{4} \exp \pth{
        T^{1/3}(\Upsilon_T + s)
        }
        }
    }\, . 
\end{align*}
From the above, we see that \eqref{3.1.eqn.anal.weak} and \eqref{3.1.eqn.anal} imply \eqref{1.4.eqn.anal.weak} and \eqref{1.4.eqn.anal} of Theorem~\ref{thm.1.1.anal}, respectively.

We now show that \eqref{3.2.eqn.anal} yields \eqref{1.5.eqn.anal}.
Let $\bar{s} := (1-\ep)^{-1}s$. Observe that 
\begin{align}
    \fk{R} &:= \E \brak{\exp\pth{- \frac{1}{4}\exp \pth{ T^{1/3} (\Upsilon_T + \bar{s})}
    }
    } 
    \nonumber \\ 
    &\leq 
    \E \brak{
    \mathds{1} \pth{ \Upsilon_T \leq -s } 
    + \mathds{1} \pth{ \Upsilon_T > -s }
    \exp\pth{- \frac{1}{4}\exp \pth{ T^{1/3} (\Upsilon_T + \bar{s})}
    }
    } 
    \nonumber \\ 
    &\leq 
    \E \brak{
    \mathds{1} \pth{ \Upsilon_T \leq -s } 
    + \mathds{1} \pth{ \Upsilon_T > -s  }
    \exp\pth{- \frac{1}{4}\exp \pth{ \ep \bar{s} T^{1/3} }
    }
    }\, .
    \label{pf.main.thm.1}
\end{align}
The second inequality follows from the fact that 
$\Upsilon_T > - s$ implies~$\Upsilon_T + \bar{s} > \ep \bar{s}$.
Continuing from \eqref{pf.main.thm.1}, we compute
\begin{align}
    \fk{R} &\leq \P(\Upsilon_T \leq - s) + 
    \exp\pth{- \frac{1}{4}\exp \pth{ \ep \bar{s} T^{1/3} }
    } \, .
    \label{3.3.eqn.anal}
\end{align}
It follows from \eqref{3.2.eqn.anal} that for all $s \geq S := S(\ep, \delta, T_0)$ and $T \geq T_0$,
\begin{align}
    \fk{R} \geq 
    \exp \pth{
    -(1+ C \ep + C' \ep) \frac{2}{15 \pi} T^{1/3} s^{5/2}
    }
    + 
    \exp \pth{- K_2 s^3} \, .
    \label{3.4.eqn.anal}
\end{align}
Here, the $C'\ep$ term 
appears because $\bar{s}^{5/2} \leq s^{5/2}(1+ C' \ep)$ for some constant $C' >0$. 
We now note that there exists a constant $S' := S' ( \ep, \delta,  T_0)$ such that for all $s \geq S'$ and $T \geq T_0$,
\begin{align}
    \exp\pth{ \ep \bar{s} T^{1/3} } 
    &\geq 
    T^{1/3} \frac{2s^{5/2}}{15 \pi} - \log \ep,
    \ \text{and thus} \nonumber
    \\ 
    \exp \pth{ -
    \exp\pth{  \ep \bar{s} T^{1/3} } 
    }
    &\leq \ep \exp \pth{
    -\frac{2}{15\pi} T^{1/3} s^{5/2}
    }\,.
    \label{3.5.eqn.anal}
\end{align}
Solving for $\P(\Upsilon_T \leq -s)$ in \eqref{3.3.eqn.anal} and 
substituting the lower bound \eqref{3.4.eqn.anal} on $\fk{R}$ and the upper bound \eqref{3.5.eqn.anal}  on $\exp \pth{ -
    \exp\pth{  \ep \bar{s} T^{1/3} } 
    }$
yields, for all $s \geq \max\{S, S'\}$ and for all $T \geq T_0$, 
\begin{align*}
    \P\pth{ \Upsilon_T \leq -s }
    \geq 
    (1-\ep)\exp \pth{
    -(1+ (C + C') \ep) \frac{2}{15 \pi} T^{1/3} s^{5/2}
    }
    + 
    \exp \pth{- K_2 s^3} \, .
\end{align*}
The multiplicative factor $(1-\ep)$ can be absorbed into the  $(1+ (C+ C') \ep)$ factor on the right-hand side above.  
Finally, taking $C := C + C'$ yields the right-hand side of \eqref{1.5.eqn.anal}, thus completing the proof of Theorem~\ref{thm.1.1.anal}. \qed

\subsection{Proof of Proposition \ref{prop.3.1.anal}}
\label{section.proof.prop.3.1.anal}
As above, let $(\a_1 > \a_2 > \dots) $ denote the GOE point process.
Define
\begin{align}
    I_s(x) &:=  \dfrac{1}{\sqrt{1+\exp(T^{1/3}(x+s))}}, 
    \text{ and}
    \label{def_I}\\
    J_s(x) &:= - \log (I_s(x)) = \frac{1}{2} \log(
    1+ \exp (T^{1/3}(x+s))
    ) \,. \label{def_J}
\end{align}
We now give upper and lower bounds on $\E_{\GOE} \brak{\prod_{k=1}^{\infty} I_s(\a_k)}$.  These bounds and Proposition \ref{Intro.ltf} allow us to complete the proof of Proposition \ref{prop.3.1.anal}.

\begin{proposition}
\label{prop.4.2.anal}
Fix any $\eta >0$, $\ep \in (0,1/3)$, $\delta \in (0,1/4)$, and $T_0 >0$. There exist positive constants $S_0 := S_0 (\eta, \ep, \delta, T_0)$,  $C:=C(T_0)$, $K_1 := K_1(\ep, \delta, T_0) >0$, and $K_2 := K_2(T_0) > 0$ 
such that for all $s \geq S_0$ and $T \geq T_0$, we have 
\begin{align}
    \E_{\GOE} \brak{ \prod_{k=1}^{\infty} I_s (\a_k) }
    \leq 
    e^{
    -\frac{2(1- C\ep)}{15 \pi} T^{1/3} s^{5/2}
    }
    +
    e^{ -\frac{\ep}{2} sT^{1/3} -\eta s^{3/2}}
    +e^{-\frac{1-C\ep}{24}s^3}
    \label{4.5.anal.eqn}
\end{align}
and 
\begin{align}
    \E_{\GOE}
    \brak{ \prod_{k=1}^{\infty} I_s (\a_k) } 
    \geq 
    e^{
    -\frac{2(1+C\ep)}{15\pi}T^{1/3} s^{5/2}
    }
    + 
    e^{- K_2 s^3} \,.
    \label{4.4.anal.eqn.weak}
\end{align}
Assuming Conjecture~\ref{conjecture}, we have the stronger upper bound
\begin{align}
    \E_{\GOE} \brak{ \prod_{k=1}^{\infty} I_s (\a_k) }
    \leq 
    e^{
    -\frac{2(1- C\ep)}{15 \pi} T^{1/3} s^{5/2}
    }
    +
    e^{ -\frac{\ep}{2} sT^{1/3} -K_1 s^{3- \delta}}
    +e^{-\frac{1-C\ep}{24}s^3}
    \label{4.4.anal.eqn}
\end{align}

\end{proposition}

We complete the proof of~\eqref{4.4.anal.eqn.weak} and~\eqref{4.4.anal.eqn} in Section~\ref{section.upper.bound.proof},
and the proof of\eqref{4.5.anal.eqn} in Section~\ref{section.lower.bound.proof}.

\begin{proof}[Proof of Proposition \ref{prop.3.1.anal}]
This follows immediately from Proposition~\ref{Intro.ltf} and Proposition~\ref{prop.4.2.anal}.
\end{proof}

\section{The GOE point process}
\label{section.goe}

Proposition \ref{prop.4.2.anal} reduces our problem to a question about the GOE point process. In this section, we formally define this process and examine results pertaining to the statistics of the process, such as the distribution of the GOE points, the typical locations of individual points, and deviations away from these typical locations. The results developed here connect the GOE point process to the \emph{stochastic Airy operator} (see Section \ref{subsection.SAO}) and will be crucial to the proofs that follow.

\subsection{First notions}
\label{subsection_GOEPP}

The \textit{GOE point process}, denoted by $(\a_1 > \a_2 > \dots )$ or $\chi^{\GOE}$, is a \textit{simple Pfaffian point process} on $(\R, \mc{B}(\R), \mu)$, where here $\mu$ denotes Lebesgue measure. We define this object now. 
%
%
%
%
We first define point processes via random point configurations (see, for instance, \cite[Section~4.2.1]{AGZ10}). 
Give $\R$ the Borel sigma algebra $\mc{B}(\R)$ equipped with a positive Radon measure $\mu$ (not necessarily Lebesgue). 
Let $\rm{Conf}(\R)$ denote the space of \textit{configurations} of $\R$, that is, discrete subsets. 
For any $B \in \mc{B}(\R)$ and $X \in \rm{Conf}(\R)$, let $N_B(X) := \#\{B \cap X\}$.
Endow $\rm{Conf}(\R)$ with the sigma algebra $\Sigma$ generated by the cylinder sets $C_n^B := \{X \in \rm{Conf}(\R) : N_B(X) = n \}$, for $ n \in \Z^+$.
A \textit{point process} is a probability measure $\nu$ on $(\rm{Conf}(\R), \Sigma)$. \cite[Lemma~4.2.2]{AGZ10} shows that a random configuration $X$ with distribution $\nu$ can be associated to a non-negative integer-valued random measure $\chi$ on $(\R, \mc{B}(\R), \mu)$ 
such that 
\[
\chi(B) = N_B(X)\,,
\]
and this random measure $\chi$  will also be referred to as the point process when clear.
A point process is called \textit{simple} if $\mu( e \in \R : \chi(\{e\}) >1 ) = 0$.
Intuitively, a simple point process $\chi$ evaluated on a Borel set $B$ counts the number of points contained in $B$ of the designated random configuration.

Now, for $k \geq 1$, consider the measure $\mu_k$ on $\R^k$ such that for disjoint Borel sets $B_1, \dots, B_k \in \mc{B}(\R)$, 
\[
\mu_k(B_1 \times \dots \times B_k) = 
\E_{\nu} \brak{
\# \left \{
\text{$k$-tuples of distinct points $x_1 \in X \cap B_1, \dots, x_k \in X \cap B_k$} \right \}
}\,.
\]
Assuming that $\mu_k$ is absolutely continuous with respect to $\mu^{\otimes k}$, we define the \textit{$k$-point correlation function} $\rho_k$ of $\chi$ to be the Radon-Nykodym derivative of $\mu_k$ with respect to $\mu^{\otimes k}$. This is a locally integrable function $\rho_k:  \R^k \to [0,\infty)$ such that, for measurable functions $f: \R \to \C$, we have
\begin{align}
\E_{\nu} \brak{ 
\sum_{(x_1, \dots, x_k) \in X^k} f(x_1) \dots f(x_k)
} 
= 
\int_{\R^k} \rho_k(x_1, \dots, x_k) f(x_1) \dots f(x_k) ~\rm{d}\mu^{\otimes k} \,.
\label{eqn:def-correlation-fns}
\end{align}
Here, $X$ is a random configuration with law $\nu$.
One might note that our definition of $\rho_k$ does not specify its value on points $(x_1, \dots, x_k)$ where $x_i = x_j$ for some $i \neq j$. On such points, we set $\rho_k = 0$; to understand the reasoning behind this, see \cite[Remark~4.2.4]{AGZ10}.
We call $\chi$ a \textit{Pfaffian point process} if there exists a $2\times 2$ skew-symmetric matrix-kernel $K:\R^2 \to M_2(\C)$ such that
\[
\rho_k(x_1, \dots, x_k) = \Pf [K(x_i, x_j)]_{i,j=1}^k\,,
\]
where $\Pf$ denotes the Pfaffian.

The GOE point process is the simple Pfaffian point process  with  correlation kernel $K^{\GOE}$, whose explicit form can be found, for instance, in \cite[Definition~6.1]{BBCW17} (we will not need the explicit form of $K^{\GOE}$  here). The GOE point process can be constructed as the limiting point process of the largest eigenvalues of the GOE $n\times n$ ensemble under so-called edge scaling, that is, centering by $2\sqrt{n}$ and scaling by $n^{1/6}$. We write $\chi^{\GOE}$ to denote the associated random measure and
$\rho_k^{\GOE}$ to denote the $k^{\mathrm{th}}$ correlation function of the GOE point process.
We  also write
$(\a_1 > \a_2 > \dots )$ to denote the random configuration of GOE points.

Proposition \ref{Intro.ltf} and the work achieved in Section \ref{section.proof.thm.1.1.anal} show that studying the GOE point process can serve as a proxy for studying the lower tail of the half-space KPZ equation. Theorem~\ref{1.3_anal_prop} establishes a basic statistic of the GOE point process: its expectation on the interval $[-s, \infty)$, for any $s>0$. We now prove this theorem.

\begin{proof}[Proof of Theorem~\ref{1.3_anal_prop}]
Note that for any point process $\chi$ with one-point correlation function $\rho_{1}$, we have on any interval $I \subseteq  \R$,
\begin{align*}
    \mathbb{E}\brak{\chi(I)} &= \int_I \rho_{1}(x) ~dx \,.
\end{align*}
Thus, we have
\begin{align}
    \E_{\GOE} \brak{\chi^{\GOE}(\fk{B}_1(s))}
    = 
    \int_{-s}^{\infty} \rho_1^{\GOE}(x) ~dx \, ,
    \label{eqn:goe-corr-fn}
\end{align}
for $s > 0$.
Let $\rho_{1}^{\GUE}$  denote the one-point correlation function for $\chi^{\GUE}$ .
From Equations~$(7.67)$ and~$(7.147)$ of \cite{For10}, we have the relation\footnote{\cite[Equation~7.147]{For10} writes this equation with ``$K^{\mathrm{soft}}(x,x)$" instead of $\rho_1^{\GUE}(x)$, as we have here, where $K^{\mathrm{soft}}(\cdot,\cdot)$ is defined in \cite[Equation~7.12]{For10}. Our expression follows from \cite[Equation~7.67]{For10}, which shows that $K^{\mathrm{soft}}(x,x) = \rho_1^{\GUE}(x)$, for any $x \in \R$.}
\begin{align}
    \rho_1^{\GOE}(x) = \rho_1^{\GUE}(x) 
    + \frac{1}{2} \Ai(x) \big ( 1 - \int_{x}^{\infty} \Ai(t) ~dt \big ) \, ,
    \label{eqn:E-rho1-1}
\end{align}
where $\Ai(\cdot)$ denotes the Airy function. Since $\int_{-\infty}^{\infty} \Ai(t) ~dt = 1$ (\cite[Equation~9.10.11]{NIST}), we may write \eqref{eqn:E-rho1-1} as 
\begin{align}
    \rho_1^{\GOE}(x) = \rho_1^{\GUE}(x) 
    + \frac{1}{2} \Ai(x)  \int_{-\infty}^{x} \Ai(t) ~dt \, .
    \label{eqn:goe-corr-main}
\end{align}
Now, \cite[Equation~7.69]{For10}, \cite[Equation~9.7.9]{NIST}, and \cite[Equation~9.10.6]{NIST} yield the following asymptotic expansions for $\rho_1^{\GUE}(x)$, $\Ai(x)$, and $\int_{-\infty}^{x} \Ai(t)~dt$  respectively, as  $x \to -\infty$:
\begin{align}
    \rho^{\GUE}_{1}(x) 
    &=
    \frac{\sqrt{-x}}{\pi} - \frac{\cos \pth{\frac{4}{3} (-x)^{3/2}}}{4\pi (-x)} + \mc{O}\big((-x)^{-5/2}\big) \, , 
    \label{eqn:gue-asymp}
    \\
    \Ai(x) 
    &=
    \frac{\cos \Big( \frac{2}{3}(-x)^{3/2} -\frac{\pi}{4}\Big)}{ \sqrt{\pi} (-x)^{1/4}}
    + \O\Big((-x)^{-7/4}\Big) \,, \text{ and} 
    \label{eqn:ai-asymp}
    \\
    \int_{-\infty}^{x} \Ai(t)~dt
    &= 
    \frac{\cos \Big( \frac{2}{3}(-x)^{3/2}+ \frac{\pi}{4} \Big)}{\sqrt{\pi}(-x)^{3/4}}   + \O\Big( (-x)^{-9/4} \Big) \, .
    \label{eqn:int-ai-asymp}
\end{align}
Substituting \eqref{eqn:gue-asymp}--\eqref{eqn:int-ai-asymp} into \eqref{eqn:goe-corr-main} yields
\begin{align*}
    \rho_1^{\GOE}(x)
    =
    \frac{\sqrt{-x}}{\pi} + \O \pth{ (-x)^{-5/2} } \, ,
\end{align*}
as $x \to -\infty$ (note that the cosine terms above cancel with one another after substitution into \eqref{eqn:goe-corr-main}). It follows that 
\begin{align}
    \int_{-s}^{-1} \rho_1^{\GOE} (x) ~dx = \frac{2}{3 \pi}s^{3/2} + \fk{D}_1(s) \, ,
    \label{eqn:rho-goe-int-s}
\end{align}
where $\fk{D}_1(s)$ satisfies $\sup_{s > 0 } \abs{\fk{D}(s)} < \infty$.

Next, because $\rho_1^{\GUE}(x)$,  $\Ai(x)$, and $\int_{-\infty}^x \Ai(t)~dt$ are bounded over  $x \in [-1,0]$, we have
\begin{align}
    \int_{-1}^{0} \rho_1^{\GOE} (x) ~dx = \fk{D}_2 \, ,
    \label{eqn:rho-goe-int-0-1}
\end{align}
for some constant $\fk{D}_2 <\infty$.  

It remains to handle the integral of $\rho_1^{\GOE}(x)$ over $x \in [0, \infty)$. \cite[Equation~7.72]{For10} states that
\begin{align*}
    \rho_1(x) = e^{-4x^{3/2}/3} \pth{1+ o(1)}\,,
\end{align*}
and thus we have $\int_{0}^{\infty} \rho_1^{\GUE}(x)~dx = \fk{D}_3$, for some constant $\fk{D}_3$. Recall that $\Ai(x) \geq 0$ for $x \geq 0$. It then follows from \eqref{eqn:goe-corr-main} and the triangle inequality that
\begin{align}
    \abs{ \int_0^{\infty} \rho_1^{\GOE}(x) ~dx}
    &\leq 
    \abs{\fk{D}_3}
    + \frac{1}{2} \int_0^{\infty} \Ai(x)
    \abs{\int_{-\infty}^x \Ai(t)~dt} ~dx \, .
    \label{eqn:rho-goe-int>0}
\end{align}
Since $\int_{-\infty}^{\infty} \Ai(t)~dt = 1$, $\int_{-\infty}^0 \Ai(t)~dt = 2/3$ (\cite[Equation~9.10.11]{NIST}), and $\Ai(t) \geq 0$ for $t \geq 0$, we have $\abs{\int_{-\infty}^x \Ai(t)~dt} \leq \abs{\int_{-\infty}^{\infty} \Ai(t)~dt} = 1$ for all $x \geq 0$. It then follows from \eqref{eqn:rho-goe-int>0} and the identity  $\int_{0}^\infty \Ai(t)~dt = 1/3$ that
\begin{align}
    \int_0^{\infty} \rho_1^{\GOE}(x) ~dx  = \fk{D}_4 \, ,
    \label{eqn:rho-goe-int>0-final}
\end{align}
for some constant $\fk{D}_4$. Combining equations~\eqref{eqn:goe-corr-fn},
\eqref{eqn:rho-goe-int-s},
\eqref{eqn:rho-goe-int-0-1}, and \eqref{eqn:rho-goe-int>0-final} yields
\begin{align}
    \E_{\GOE}\brak{\chi^{\GOE}( \fk{B}_1(s))}
    = \frac{2}{3\pi}s^{3/2} + D_1(s) \,,
\end{align}
where 
$D_1(s) = \fk{D}_1(s) + \fk{D}_2 + \fk{D}_4$, and therefore clearly satisfies $\sup_{s > 0} \abs{D(s)} < \infty$. Thus, we have \eqref{1.3_exp}.
\end{proof}

\subsection{The \texorpdfstring{$\beta$}{} stochastic Airy operator}
\label{subsection.SAO}

We now apply and enhance the tools  developed in \cite[Section~4.3]{CG18} to connect the GOE point process with the eigenvalues of the \textit{stochastic Airy operator} $\mc{H}_{\beta}$ with $\beta =1$.
Observed in \cite{ES07} and proved in \cite[Theorem~1.1]{RRV11}, Proposition \ref{rrv11_thm} gives an equivalence in distribution between the eigenvalues of $\mc{H}_{\beta}$ and the negatives of the GOE points. 
Proposition \ref{prop4.5_cg18} below was proved in \cite[Proposition~4.5]{CG18}, and establishes a uniform bound on the deviations of the  (random) $\mc{H}_{\beta}$ eigenvalues from the eigenvalues of the (deterministic) \textit{Airy operator}, and Theorem~\ref{1.6_anal} establishes the same uniform bound on deviations of the GOE points from these deterministic eigenvalues.  
Finally, Proposition \ref{airy_eval_bound}, which was proved in \cite{MT59}, approximates the location of each eigenvalue of the Airy operator.

We now define the stochastic Airy operator through the theory of Schwartz distributions.

\begin{definition}[stochastic Airy operator]
Let $D := D(\R^+)$ denote the space of distributions, i.e., 
the continuous dual of the space of smooth, compactly supported test functions
equipped with the topology of uniform convergence of all derivatives on compact sets. 
All formal derivatives of any continuous function $f$ are distributions, with action on any test function $\phi \in C_0^{\infty}$ given by integration by parts as follows:
\[
\prec \phi, f^{(k)}(x) \succ 
:=
(-1)^k \int f(x) \phi^{(k)}(x) ~dx \,,
\]
where $\prec \cdot, \cdot \succ$ is notation not to be confused with the $L^2$ inner product $\ip{\cdot, \cdot}$.
In particular, since Brownian motion $B$ is a random continuous function, its formal derivative $B'$ is a random element of $D$.
The $\beta > 0$ \textbf{stochastic Airy operator}
is a random linear map 
\[
\mc{H}_{\beta}: H_{\rm{loc}}^1 \to D
\]
such that 
\[
\mc{H}_{\beta} f = -f^{(2)} + xf + \frac{2}{\sqrt{\beta}} fB' \,,
\]
where $H^1_{\rm{loc}}$ is the space of functions $f: \R^+ \to \R$ such that for any compact $I$, $f' \mathds{1}(I) \in L^2$.
Though $D$ is only closed under multiplication by smooth functions and $f \in H_{\rm{loc}}^1$,
we make sense of $fB'$ as the derivative  of
$\int_0^y fB' ~dx := -\int_0^y Bf'~dx + f(y) B_y - f(0) B_0$.
The \bf{Airy operator} $\mc{A} := -\p_x^2 + x$ is the non-random part of $\mc{H}_{\beta}$.

To define the eigenvalues/eigenfunctions of $\mc{H}_{\beta}$, we define the Hilbert space $L^*$ with norm
\[
\norm{f}_*^2 =
\int_0^{\infty} \pth{(f')^2 + (1+x) f^2} ~dx \,, 
\ \ \ \ \
L^* := 
\{
f : f(0) = 0, \norm{f}_* < \infty
\} \, .
\]
We say a pair $(f, \Lambda) \in L^* \times \R$ is an eigenfunction/eigenvalue pair for $\mc{H}_{\beta}$ if
$\mc{H}_{\beta}f = \Lambda f$.
\end{definition} 

The following is a special case of \cite[Theorem~1.1]{RRV11}, namely, the $\beta =1$ case.\footnote{The result is proved for any $\beta$: under edge scaling, the $k$ largest eigenvalues of the $n\times n$ Hermite $\beta$-ensemble converge jointly in distribution to the smallest $k$ eigenvalues of $\mc{H}_{\beta}$ as $n \to \infty$.}

\begin{proposition}[{\cite[Theorem~1.1]{RRV11}}]
\label{rrv11_thm}
\label{rrv11_thm_cor}
Let $(\Lambda_1 < \Lambda_2 < \dots)$ denote the eigenvalues of $\mc{H}_{1}$, and recall that $(\a_1 > \a_2 > \dots)$ denotes the GOE point process. Then for any $k \geq 1$,  we have
\begin{align}
   (-\a_1, \dots, -\a_k) \ \overset{(\d)}{=} \
     (\Lambda_1, \dots, \Lambda_{k}) \,. \label{rrv11_4.4}
\end{align}
\end{proposition}

\cite{RRV11} and \cite{Vir14} show that there exists a  random band with uniform width $C_{\ep}$ around each eigenvalue of the Airy operator such that  each eigenvalue of $\mc{H}_{\beta}$ is contained in the band around the corresponding Airy operator eigenvalue.

\begin{proposition}[{\cite[Proposition~4.5]{CG18}}]
\label{prop4.5_cg18}
Denote the eigenvalues of the Airy operator $\mc{A}$  by $(\ld_1 < \ld_2 < \dots)$ and the eigenvalues of $\mc{H}_{\beta}$ by $(\Lambda_1^{\beta} < \Lambda_2^{\beta} < \dots)$. 
For any $\ep \in (0,1)$, define the random variable $C_{\ep}$ as the smallest real number such that for all $k \geq 1$, 
\[
(1-\ep) \ld_k - C_{\ep} \leq \Lambda_k^{\beta} \leq (1+ \ep) \ld_k + C_{\ep}. 
\]
Then for all $\ep, \delta \in (0,1)$, there exist positive constants $S_0 := S_0(\ep, \delta)$, and $\kappa :=  \kappa(\ep, \delta)$ such that for all $s\geq S_0$,
\begin{align}
    \P \pth{C_{\ep} \geq \frac{s}{\sqrt{\beta}}} \leq \kappa \exp \pth{-\kappa s^{1-\delta}}.
    \label{4.6.eqn.anal}
\end{align}

\end{proposition}

Proposition~\ref{prop4.5_cg18} gives an exponential upper-tail bound on $C_{\ep}$ that will be crucial to our proof of Theorem~\ref{1.5_anal}.
Note that Theorem~\ref{1.6_anal} follows immediately from Propositions~\ref{rrv11_thm} and~\ref{prop4.5_cg18}.

To prove Theorem~\ref{1.5_anal}, we will also need the following results on the approximate location of eigenvalues of the Airy operator $\mc{A} = -\p_x^2+ x$.

\begin{proposition}[\cite{MT59}] 
\label{airy_eval_bound}
If the eigenvalues of the Airy operator $\mc{A}$ are denoted by $(\lambda_1 < \ld_2 <\dots)$, then for all $n \geq 1$, we have 
\begin{align}
    \ld_n = \pth{
    \frac{3\pi}{2}\pth{n - \frac{1}{4} + \mc{R}(n)}
    }^{2/3}, \label{1.5_anal_evalest}
\end{align}
where for some large constant $K \in \R$, we have 
\[
\abs{\mc{R}(n)} \leq K/n.
\]

\end{proposition}

\begin{corollary}\label{1.5_anal_expchi}
For any $T \in \R_{\geq 0}$, define $k := k(T) = \#\{n: \ld_n \leq T \}$.
We have 
\[
k = \frac{2}{3\pi}T^{3/2} + C_1(T) \,,\]
where $\sup_{x>0} \abs{C_1(x)} < 1$; 
thus,
\begin{align}
    k- \E\brak{\chi^{\GOE}[-T, \infty)} = \mc{O}_{T}(1) \, .
\end{align}

\begin{proof}
From \eqref{1.5_anal_evalest}, it is clear that $k = \floor{x}$, where $x \in \R_{\geq 0}$ satisfies 
\begin{align}
    T &= \pth{
    \frac{3\pi}{2}\pth{x - \frac{1}{4} + \mc{R}(x)}
    }^{2/3} \, .
\end{align}
Solving for $x$ gives 
\begin{align}
    x &= \frac{2}{3\pi} T^{3/2} + \frac{1}{4} + \mc{R}(x) \, .
\end{align}
Recall from Proposition~\ref{airy_eval_bound} that $\abs{\mc{R}(x)} \leq K/x$. As $T$ approaches $\infty$, we have $x \sim \frac{2}{3\pi}T^{3/2}$, and thus  $k$ will simply be the the closest integer to $\frac{2}{3\pi} T^{3/2} + \frac{1}{4}$. 
From the expression $\E\brak{\chi^{\GOE}[-T, \infty)} = \frac{2}{3\pi}T^{3/2} + D_1(T)$ given by Theorem~\ref{1.3_anal_prop}, the corollary follows. 
\end{proof}
\end{corollary}

\section{The cumulant generating function for \texorpdfstring{$\chi^{\GOE}$}{}}
\label{section.thm.1.7}
The proof of Theorem~\ref{1.4.final}, which makes up the contents of Section \ref{section.thm.1.4.proof}, will boil down to estimating the cumulant generating function for $\chi^{\GOE}$,
\begin{align*}
    F_1(s, v) &:= \E \brak{ \exp \pth{ - v \chi^{\rm{\GOE}} \pth{ [s,\infty) }}} \, .
\end{align*}
The main result of this section is Theorem~\ref{1.7.rewrite.F_1}, which connects $F_1(s,v)$ to the distribution function of the largest eigenvalue of the \textit{thinned GOE point process} via a \textit{Fredholm Pfaffian}. Theorem~\ref{1.7.rewrite.F_1} is a major input towards
Corollary~\ref{cor:bb17-fixed} and Theorem~\ref{1.7.final}, which provide the needed bounds on $F_1(s,v)$ to prove Theorem~\ref{1.4.final} in Section~\ref{section.thm.1.4.proof}.

\subsection{The thinned GOE point process and the Painlev\'e II equation}
\label{subsection.thinned.goe.painleve}

Theorem~\ref{1.7.rewrite.F_1} equates $F_1(s,v)$ to the distribution function $\mc{F}_1(s,v)$ of the largest particle $\a_1(\gamma)$  of the \textit{thinned GOE point process with parameter} $\gamma:= 1- e^{-v}$.
This is the point process obtained by independently removing each particle of the GOE point process (see Section \ref{section.goe}) with probability $1 - \gamma$. 
We may similarly define the \textit{thinned GUE point process} and the distribution function $\mc{F}_2(s,v)$ of the largest particle of the thinned GUE point process with parameter $\gamma$. Note that, like the GOE point process, the thinned GOE point process is simple and Pfaffian. To see that it is Pfaffian, let $\{Y_i\}_{i \in \N}$ be a sequence of i.i.d. Bernoulli random variables such that $\P(Y_1 = 1) = \gamma$. 
Let $\nu^{\GOE}$ and $\nu^{\mathrm{thin}}$ be the laws on $\mathrm{Conf}(\R)$ associated to the GOE and thinned GOE point processes respectively, and let $X$ and $\hat{X}$ be random configurations with laws $\nu^{\GOE}$ and $\nu^{\mathrm{thin}}$ respectively. Then, for a measurable function $f:  \R \to \C$, we have 
\begin{align*}
    \E \brak{ 
\sum_{(x_1, \dots, x_k) \in \hat{X}^k} f(x_1) \dots f(x_k)
} 
&= \E \brak{ 
\sum_{(x_1, \dots, x_k) \in X^k} \prod_{i=1}^k  f(x_i) Y_i
} 
= \gamma^k \E \brak{ 
\sum_{(x_1, \dots, x_k) \in X^k} \prod_{i=1}^k  f(x_i) 
} \, ,
\end{align*}
where the last equality follows from the independence of the $Y_i$ from each other and from the GOE point process. We then have from \eqref{eqn:def-correlation-fns} that, for any $k \geq 1$, 
\[
\rho_k^{\mathrm{thin}} = \gamma^k \rho_k^{\GOE} \,,
\]
where $\rho_k^{\mathrm{thin}}$ denotes the $k^{\mathrm{th}}$ correlation functions for the thinned GOE point process. Furthermore, it follows that the correlation kernel for the thinned GOE point process is $\gamma K^{\GOE}$. 

Proposition \ref{1.7.prop.bb17} below gives a formula for $\mc{F}_1(s,v)$ in terms of $\mc{F}_2(s,v)$ and a certain integral of the Ablowitz-Segur (AS) solution $u_{\rm{AS}}(\cdot,\gamma)$ to the \textit{Painlev\'e II equation}. Recall from Section~\ref{subsec:uas} that
$u_{\rm{AS}}$ is a  one-parameter family of solutions to 
\begin{align}
    u_{\rm{AS}}(s,\gamma)'' &= x u_{\rm{AS}}(s,\gamma) + 2 u_{\rm{AS}}^3(s, \gamma) \nonumber
\end{align}
with boundary coundition 
\begin{align}
    u_{\rm{AS}}(s, \gamma) &= \sqrt{\gamma} \frac{s^{-1/4}}{2 \sqrt{\pi}}
    e^{-\frac{2}{3}s^{3/2}} (1+ o(1)), \text{ as $s \to \infty$.} \nonumber
\end{align}
Proposition \ref{1.7.prop.bb17} comes from \cite[Proposition~1.1]{BB17}, though in \cite[Remark~1.2]{BB17}, the authors note that the formula can be obtained via some combination of results in \cite{BdCP09}. 

\begin{proposition}[{\cite{BB17}}]
\label{1.7.prop.bb17}
For any $s \in \R$ and $v> 0$, we have 
\begin{align}
    \mc{F}_2(s,v) &= 
    \exp \pth{
    -\int_s^{\infty} (t-s) u_{\rm{AS}}^2(t;\gamma) ~dt
    }
    \label{1.7.BB17.painleve2.connection}
\end{align}
and 
\begin{align}
     \mc{F}_1(s,v) &= \sqrt{\mc{F}_2(s, 2v)} \sqrt{1+\dfrac{\cosh \mu(s, \gamma_2) - \sqrt{\gamma_2} \sinh \mu(s, \gamma_2) -1 }{2- \gamma}} \,,
     \label{1.7.thin.dist.fn.eqn}
\end{align}
where $\gamma$, 
$\mu(s, \gamma_2)$ and $\gamma_2$ are defined as in the statement of Theorem \ref{1.7.rewrite.F_1.corollary}.
\end{proposition}

Let 
$F_2(s,v) := \E \brak{\exp \pth{-v \chi^{\rm{Ai}}\pth{ [s, \infty) }
}}$ be
the cumulant generating function of the GUE point process. One of the major technical achievements of \cite{CG18} is given below as Proposition~\ref{prop.thm.1.7.cg.18}, which bounds $F_2(s,v)$ by equating it to $\mc{F}_2(s,v)$ and then using the connection to the Painlev\'e II equation given by \eqref{1.7.BB17.painleve2.connection} to conduct a fine analysis.

\begin{proposition}[{\cite[Theorem~1.7]{CG18}}] \label{prop.thm.1.7.cg.18}
For all $v$ and $s$ in $\R$, we have 
\begin{align}
    F_2(s,v) = \mc{F}_2(s,v) = \exp \bigg( - \int_{-s}^{\infty}(x+s) u_{\AS}^2(x;\gamma)~dx \bigg)\,,
\end{align}
where $\gamma := \gamma(v) =  1-e^{-v}$.
Furthermore, for any fixed $\delta \in (0, \frac{2}{5})$,  as $s$ goes to $\infty$,
\begin{equation}
    \log F_2(-s,s^{\frac{3}{2}-\delta}) \leq -\frac{2}{3\pi}s^{3- \delta} + \mc{O}(s^{3- \frac{13\delta}{11}}).
    \label{prop.thm.1.7.cg.18.eqn}
\end{equation}
\end{proposition}

%
\subsection{Fredholm Pfaffians}
\label{subsection.fredholm.pfaffian}
The Fredholm Pfaffian was first defined in \cite{Rai00}; the definition reproduced below comes from \cite{BBCS16}.

\begin{definition} 
\label{def:fred-pf}
Let $\mu$ be a reference measure on $\R$, and 
let $K(x,y)$ be a $2\times 2$ matrix-valued skew-symmetric kernel on $\mathbb{R}^2$. 
Define 
\[
J(x,y) = \mathds{1}_{(x=y)} \begin{pmatrix}
0 & 1 \\ -1 & 0
\end{pmatrix}, ~\forall x, y \in \R\,.
\]
Then the \textbf{Fredholm Pfaffian} of $K$ is defined by the series expansion 
\begin{align}
    \Pf(J+K)_{\bb{L}^2(\R, \mu)} &:=
    1+ \sum_{k=1}^{\infty} \frac{1}{k!} \int_{\R} \cdots \int_{\R} \Pf \pth{K(x_i, x_j)_{i,j=1}^k} d\mu^{\otimes^k}(x_1, \dots x_k)\,, \label{pfaff_fred}
\end{align}
provided that the series converges.
\end{definition}

Let the measure $\nu$ on $(\mathrm{Conf} (\R), \Sigma)$ be a  Pfaffian point process on 
$(\R,  \mc{B}(\R), \mu)$ with matrix kernel $K$, and let $X$ denote a random configuration with law $\nu$ (see Section~\ref{subsection_GOEPP} for definitions of these objects).
For any measurable function $f: \R \to \C$, \cite[Theorem~8.2]{Rai00} gives the identity
\begin{align}
    \E_{\nu} \brak{ \prod_{x \in X} (1 + f(x))} = \Pf(J+K)_{\bb{L}^2(\R, f \mu)}\,,
    \label{exp_prod_pfaff}
\end{align}
whenever both sides converge absolutely.
This identity can be applied to obtain a Fredholm Pfaffian representation for $F_1$. Consider the GOE point process, which  we recall is a Pfafian point process on $(\R, \mc{B}(\R), \mu)$, where $\mu$ denotes the Lebesgue measure. Recall also that we write $(\a_1 > \a_2 > \dots)$ to denote the random configuration of GOE points. For any $s \in \R$ and $v \geq 0$, taking
$
f(x) := e^{-v \mathds{1}(x \geq s)} - 1
$
in \eqref{exp_prod_pfaff} yields
\begin{align}
    F_1(s,v)
    &= \E_{\GOE} \brak{ \prod_{\a_i} e^{-v \mathds{1}(\a_i \geq s)} }
    = \Pf(J+K^{\GOE})_{\bb{L}^2(\R, f\mu)}\,, \label{barF}
\end{align}
provided that the right-hand side above converges absolutely. The absolute convergence is shown in the proof of Theorem~\ref{1.7.rewrite.F_1} below.

\begin{theorem} \label{1.7.rewrite.F_1}
Let $\mc{F}_1(s,v)$ denote the distribution function of the largest particle of the thinned GOE point process $\a_1(\gamma)$ with parameter $\gamma := 1- e^{-v}$, where $s\in \R$ and  $v \geq 0$. Then we have
\begin{align}
    F_1(s,v)  = \Pf(J- \gamma K^{\GOE})_{\bb{L}^2([s, \infty),\mu)} = \mc{F}_1(s,v).
    \label{1.7_F_1_fredholm}
\end{align}
where $\mu$ denotes the Lebesgue measure.
\end{theorem}

\begin{proof}
We begin by demonstrating the absolute convergence of the right-hand side of \eqref{barF}, which may be expanded as
\begin{align}
    &1+ \sum_{k=1}^{\infty} \frac{1}{k!}
    \int_{\R} \cdots \int_{\R} 
    \Pf \pth{K^{\GOE}(x_i, x_j)}_{i,j=1}^k 
    \prod_{i=1}^k \pth{e^{-v \mathds{1}(x_i \geq s)} -1} d\mu^{\otimes^k}(x_1, \dots, x_k) 
    \nonumber \\
    &=1+ \sum_{k=1}^{\infty} \frac{\pth{e^{-v} -1}^k }{k!}
    \int_{[s, \infty)} \cdots \int_{[s, \infty)} 
    \Pf \pth{K^{\GOE}(x_i, x_j)}_{i,j=1}^k 
    d\mu^{\otimes^k}(x_1, \dots, x_k) \,. \label{absorb const}
\end{align}
Observe that since $v \geq 0$, $\abs{e^{-v}-1} \leq 1$. This along with the bound on $\big|\Pf \pth{K^{\GOE}(x_i, x_j)}_{i,j=1}^k \big|$  given in \cite[Proposition~4.1(i)]{Lin20} allows us to compute
\begin{align}
    &\sum_{k=1}^{\infty} \frac{\big| {\pth{e^{-v} -1}^k } \big|}{k!}
    \int_{[s, \infty)} \cdots \int_{[s, \infty)} 
    \big|\Pf \pth{K^{\GOE}(x_i, x_j)}_{i,j=1}^k \big|
    ~d\mu^{\otimes^k}(x_1, \dots, x_k) 
    \nonumber \\
    &\leq 
    \sum_{k=1}^{\infty} \frac{k^{k/2} C^k}{k!}
    \pth{  
    \int_s^{\infty} e^{-x_i^{3/2}/3} \mathds{1}_{ \{x_i \geq 0 \} }
    + (1-x)^2 \mathds{1}_{\{x < 0\}}  ~d\mu(x)
    }^k  
    \nonumber \\
    &\leq 
    \sum_{k=1}^{\infty} \frac{k^{k/2} C_s^k}{k!} < \infty \,,
\end{align}
where $C$ is a positive constant, $C_s$ is a positive constant depending only on $s$, and the above sum converges due to Stirling's formula. This establishes the Fredholm Pfaffian representation \eqref{barF} of $F_1(s,v)$.

Let us return to the expansion of the Fredholm Pfaffian in \eqref{absorb const}.
From the definition of $\Pf(A)$, we see that scaling every entry of the matrix $A$ by some constant $c$ and taking the Pfaffian is equivalent to $c^k \Pf(A)$, where $A$ is a $2k \times 2k$ matrix. Thus, from \eqref{absorb const}, we find
\begin{align}
    F_1(s,v) &=
    1+ \sum_{k=1}^{\infty} \frac{(-1)^k}{k!}
    \int_{[s, \infty)} \cdots \int_{[s, \infty)} 
    \Pf \pth{\gamma K^{\GOE}(x_i, x_j)}_{i,j=1}^k  d\mu^{\otimes^k}(x_1, \dots, x_k) \nonumber \\
    &= \Pf( J - \gamma K^{\GOE} )_{\bb{L}^2([s, \infty), \mu)} \,. \label{pfaff1.7}
\end{align}
Now, recall from the first paragraph of Section \ref{subsection.thinned.goe.painleve} that the thinned GOE point process is Pfaffian with correlation kernel $\gamma K^{\GOE}$.
Thus, the gap probability for the thinned GOE point process is  
\begin{align}
    \Pf(J- \gamma K^{\GOE})_{\bb{L}^2([s, \infty),\mu)} 
    &= \P(\a_1(\gamma) < s) \nonumber =: \mc{F}_1(s,v). 
\end{align}
Substituting this into \eqref{pfaff1.7} yields \eqref{1.7_F_1_fredholm} \,.
\end{proof}

\subsection{Proofs of Theorems \ref{1.7.rewrite.F_1.corollary} and \ref{1.7.final}}
\label{section.1.7.final.proof}
We are now ready to prove Theorem~\ref{1.7.rewrite.F_1.corollary}. Assuming Lemma~\ref{1.7.mu_est}, we will then be able to prove Theorem~\ref{1.7.final} as well. Lemma~\ref{1.7.mu_est} is proved in Section \ref{proof_of_mu_est} below.

\begin{proof}[Proof of Theorem~\ref{1.7.rewrite.F_1.corollary}]
Equation \eqref{1.7_F_1}  follows immediately from \eqref{1.7.thin.dist.fn.eqn}, Proposition~\ref{prop.thm.1.7.cg.18}, and 
Theorem~\ref{1.7.rewrite.F_1}.
\end{proof}

\begin{proof}[Proof of Theorem~\ref{1.7.final}]
Fix any $\delta \in (0,2/5)$. Take $v$ to be $\bar{v}$ (so that $\gamma = 1-e^{-\bar{v}}$ and $\gamma_2$ is equal to~$\bar{\gamma}$) in \eqref{1.7_F_1} 
. This yields
\begin{align}
F_1 \pth{ -s,\frac{1}{2}s^{3/2- \delta}}
    &= \sqrt{F_2(-s, s^{3/2- \delta})} 
    \sqrt{1 + \dfrac{\cosh \mu(-s, \bar{\gamma}) - \sqrt{\bar{\gamma}} \sinh \mu(-s, \bar{\gamma}) -1 }{2- \gamma}} \, . \label{1.7_F_1_new_v}
\end{align}
 Equation \eqref{prop.thm.1.7.cg.18.eqn} gives the following bound as $s \to \infty$:
\begin{align}
    \sqrt{F_2(-s, s^{3/2- \delta})}  
    &\leq
    \sqrt{\exp \pth{
    -\frac{2}{3 \pi} s^{3-\delta}+ \mc{O} \pth{ s^{3- \frac{13\delta}{11}} }
    }} 
    = \exp \pth{
    -\frac{1}{3 \pi}s^{3-\delta} + \mc{O} \pth{ s^{3- \frac{13\delta}{11}}}
    } \, .
    \label{1.7.mainterm}
\end{align}
Since $\bar{\gamma} \in (0,1]$ and $2-\gamma \in [1,2)$, the second term on the right-hand side of \eqref{1.7_F_1_new_v} may be crudely bounded above as $s \to \infty$ by
\[
\sqrt{C_1+C_2\exp\pth{\abs{\mu(-s, \bar{\gamma})}}}\,,
\]
for some positive constants $C_1$ and $C_2$ (independent of $s$ and $\delta$). From Lemma~\ref{1.7.mu_est} and the above display, we find that as $s \to \infty$,
\begin{equation}
    \sqrt{1 + \dfrac{\cosh \mu(-s, \bar{\gamma}) - \sqrt{\bar{\gamma}} \sinh \mu(-s, \bar{\gamma}) -1 }{2- \gamma}}
    = o(s^{3-\delta})\,.
    \label{1.7.muterm}
\end{equation}
Substituting the bounds given by  \eqref{1.7.mainterm} and \eqref{1.7.muterm} into  \eqref{1.7_F_1_new_v} yields \eqref{1.7.final.eqn}. 
\end{proof}

\subsection{Proof of Lemma~\ref{1.7.mu_est}} \label{proof_of_mu_est}
The proof of Lemma~\ref{1.7.mu_est} is given at the end of this subsection.

Throughout this subsection, as in the statement of Lemma \ref{1.7.mu_est}, we take $\delta \in (0,2/5)$ fixed. The parameter $s$ is taken to be positive, and we define $\bar{v}:= \bar{v}(s,\delta)$ 
and $\bar{\gamma} := \bar{\gamma}(s,\delta)$ as in~\eqref{def:bar-v-gamma}. Note that
$\bar{\gamma}= 1-e^{-2\bar{v}}$.
For some fixed constants $x_0> 0$ and $\zeta_0 \in (0,2\sqrt{2}/3)$ to be specified later, 
we will consider upper bounds on $u_{\AS}(x; \bar{\gamma})$ over each of the following  intervals of $x$:
\begin{enumerate}
    \item $[-s, -(\frac{2\sqrt{2}}{3}- \zeta_0 )^{-2/3} s^{1- \frac{2}{3}\delta}]$.
    \item $\big( \! \!- \!(\tfrac{2\sqrt{2}}{3}- \zeta_0 )^{-2/3} s^{1- \frac{2}{3}\delta}, -( \tfrac{2\sqrt{2}}{3})^{-2/3} s^{1-\frac{2}{3}\delta} \big ) =: \mathbf{I}_0$
    \item $\big[ \! - \!(\frac{2\sqrt{2}}{3})^{-2/3} s^{1-\frac{2}{3}\delta}, -x_0 \big)$
    \item $[-x_0, \infty)$, 
\end{enumerate}

Consider $\bar{\aleph}:= \aleph(x, \bar{\gamma})$ (where $\aleph(x,\gamma)$ was defined for general $\gamma \in [0,1)$ in  \eqref{def:aleph}).
The interval~$(1)$ corresponds to 
$\bar{\aleph} \in \bar{I}_1(\zeta_0) := [s^{-\delta}, \frac{2\sqrt{2}}{3}-\zeta_0]$, which we recall from Section~\ref{subsec:uas} is contained in the regular Boutroux region $I_1(\zeta_0) := (0, \frac{2\sqrt{2}}{3}-\zeta_0) $.
\cite[Theorem~1.10]{Bot17} gives an expansion for $u_{\AS}(x; \gamma)$ (for general $x$ and $\gamma$ such that $\aleph \in I_1(\zeta_0)$) in terms of Jacobi theta functions and elliptic integrals. 
In \cite[Section~6]{CG18}, the authors manipulate the formula from \cite[Theorem~1.10]{Bot17} into a form that is more amenable to obtaining the estimates that they seek. In our case, we only seek crude upper bounds on $u_{\AS}$, 
for which \cite[Theorem~1.10]{Bot17} and the work of \cite[Section~6]{CG18} can be combined to obtain an  upper bound of order $(-x)^{1/2}$ on $u_{\AS}(x;\gamma)$ uniformly over $\aleph \in I_1(\zeta_0)$.

\begin{lemma}
\label{lem:mu:I21}
For some constant $\zeta_0 \in (0, 2\sqrt{2}/3)$, there exists constants $S_0 > 0$ and $C>0$ such that for all $s \geq S_0$ and for all $\aleph \in  I_{1}(\zeta_0)$, 
we have 
\begin{align}
    \abs{u_{\AS}(x; \bar{\gamma})} \leq C(-x)^{1/2} \,.
    \label{eqn:pre-stokes}
\end{align}
\begin{proof}
In what follows, we rely heavily on the notation set forth at the start of \cite[Section~6.1]{CG18}--- since this notation is used only in the present proof, which is rather short, we do not redefine their notation here.
From equations~$6.1$ and~$6.2$ of
\cite{CG18}\footnote{While \cite[Proposition~6.1]{CG18} is stated for $\zeta \in (0, \sqrt{2}/3)$, it is written in a footnote that the result holds for all $\zeta \in (0,2\sqrt{2}/3)$, simply because \cite[Equation~1.26]{Bot17} holds for this wider range of $\zeta$, and \cite[Equation~6.1,6.2]{CG18} is a reformulation of \cite[Equations~1.25, 1.26]{Bot17}.} (which is a reformulation of Equations~$1.25$ and~$1.26$ of \cite{Bot17}), we see that it suffices to find appropriate bounds on 
\begin{align}
    \frac{1-\kappa}{\sqrt{1+\kappa^2}} \,, 
    \quad 
    \text{ and }
    \quad 
    \rm{cd}\pth{2(-x)^{3/2}VK\pth{\tilde{\kappa}} , \tilde{\kappa}} \,,
    \label{eqn:pre-stokes-goal}
\end{align}
where we define $\tilde{\kappa}:= \frac{1-\kappa}{1+\kappa}$. 
It follows from \cite[Equations~6.3,~6.4]{CG18} that $\kappa(\aleph)$ is bounded uniformly
over bounded regions of  $\aleph$, and so $\frac{1-\kappa}{\sqrt{1+\kappa^2}}$ is bounded uniformly over $\aleph \in I_1(\zeta_0)$.

Next, \cite[Equation~6.9]{CG18} implies that there exist $r_0 \in [0, 1)$ and $C_1 >0$ such that for all $r \leq r_0$,
\begin{align}
    \abs{\mathrm{cd}(z,r)} \leq 1+ C_1 r^2 \,.
    \label{eqn:pre-stokes-cd-bd}
\end{align}
It is shown in the proof of Lemma~6.3 of \cite{CG18} that $\tilde{\kappa}$ goes to zero as $\aleph$ goes to zero, and so there exists $\zeta_0$ sufficiently close to $2\sqrt{2}/3$ such that for all $\aleph \in (0, \frac{2\sqrt{2}}{3}-\zeta_0]$, we have $\tilde{\kappa} \leq  r_0$. 
Then from~\eqref{eqn:pre-stokes-cd-bd}, we have
\begin{align*}
    \abs{\rm{cd}\pth{2(-x)^{3/2}VK\pth{\tilde{\kappa}} , \tilde{\kappa}}} \leq  C_2,
\end{align*}
for some $C_2 >0$. Thus, both terms in \eqref{eqn:pre-stokes-goal} are bounded uniformly over $\aleph \in I_1(\zeta_0)$. Equation~\eqref{eqn:pre-stokes} then follows from \cite[Proposition~6.1]{CG18}.
\end{proof}
\end{lemma}

Taking $\zeta_0$ as in Lemma~\ref{lem:mu:I21}, it follows from \eqref{eqn:pre-stokes} that
\begin{align}
    \bigg | \int_{-s}^{-(\frac{2\sqrt{2}}{3}- \zeta_0 )^{-2/3} s^{1- \frac{2}{3}\delta}} u_{\AS}(x; \bar{\gamma}) ~dx \bigg | = \bigg | \int_{\bar\aleph \in \bar{I}_1(\zeta_0)} u_{\AS}(x; \bar{\gamma}) ~dx \bigg | \leq \mathcal{C}_1 s^{3/2}\,,
    \label{eqn:mu:1}
\end{align}
for some positive constant $\mc{C}_1$.

Interval $(2)$ corresponds to $\bar{\aleph}$ in the Stokes region $( \frac{2\sqrt{2}}{3} - \zeta, \frac{2\sqrt{2}}{3} )$, defined in Section~\ref{subsec:uas}.
Since $\mathbf{I}_0$ has length of order $s^{1- \frac{2}{3}\delta}$,
equation~\eqref{eqn:conjecture} of Conjecture~\ref{conjecture} implies that 
\begin{align}
    \int_{\mathbf{I}_0} \abs{u_{\AS}(x;\bar{\gamma})} ~dx =
    \int_{\bar{\aleph} \in ( \frac{2\sqrt{2}}{3} - \zeta, \frac{2\sqrt{2}}{3} )} \abs{u_{\AS}(x;\bar{\gamma})} ~dx
    = o(s^{3-\delta})\,. 
    \label{eqn:mu:2stokes}
\end{align}

Interval $(3)$ corresponds to 
$\bar{\aleph} \in \bar{I}_2  := [\tfrac{2\sqrt{2}}{3}, x_0^{-3/2}s^{\frac{3}{2}-\delta})
$, which we recall from Section~\ref{subsec:uas} is contained in the Hastings-McLeod region $I_2 := [\frac{2\sqrt{2}}{3},\infty)$. Over this region, we have \cite[Theorem~1.12]{Bot17}, reformulated below as Proposition~\ref{prop:Bot17-thm1.12}.

\begin{proposition}[{\cite[Theorem~1.12]{Bot17}}\footnote{It may be helpful to match the notation of \cite{Bot17} with ours. 
We have taken the parameter $f_2$ of \cite{Bot17} to be $0$. For any $\gamma \in [0,1)$, the function  $u(x|s):= u(x | (s_1,s_2,s_3))$ of \cite{Bot17} is  equal to $u_{\AS}(x; 
\gamma)$ in the special case $s = (-i \sqrt{\gamma}, 0, i \sqrt{\gamma})$, as stated in \cite[Remark~1.6]{Bot17}. The quantity $\ep$ of \cite{Bot17} is defined as $\rm{sgn}(\Im s_1)$, which is equal to $-1$ in our case. The parameter $v$ of \cite{Bot17} is also written here as $v$.
The parameter
$\aleph$ of \cite{Bot17} is defined in \cite[Equation~1.21]{Bot17} as $v (-x)^{-3/2}$, which, for $v = - \log (1-\gamma)$, matches our definition of $\aleph$.}] 
\label{prop:Bot17-thm1.12}
There exist positive constants $x_0$, $v_0$, and $c$ such that for all $-x \geq x_0$, $v:= - \log (1- \gamma) \geq v_0$, and $\aleph \in I_2$, we have 
\begin{align}
    u_{\AS}(x; \bar{\gamma}) = 
    - \sqrt{- \frac{x}{2}} 
    \Big (
    1 - \frac{e^{\frac{2}{3}\sqrt{2}(-x)^{-3/2} - v}}{\pi ( -x)^{3/4} 2^{5/4}} + J_2(x,s)
    \Big )\, ,
    \label{eqn:bot17-thm1.12}
\end{align}
where $|J_2(x,s)| \leq c(-x)^{-3/2}$. 
\end{proposition}

Take $\gamma = \bar{\gamma}$ in Proposition~\ref{prop:Bot17-thm1.12} so that $v = 2\bar{v}$ (where $\bar{v}$ was defined at the start of this subsection), and let $x_0$ be as in the proposition. Consider $S_0:= S_0(\delta)$  such that $S_0^{1- \frac{2}{3}\delta} > x_0$ and $S_0^{ \frac{3}{2}-\delta} \geq v_0$. 
Then for any $s \geq S_0$ and $x$ in interval (3) (equivalently, $\bar{\aleph} \in \bar{I}_2$), we have $-x \geq x_0$ and $2\bar{v} \geq v_0$. Thus, the hypotheses of the proposition are satisfied, and so  there exists a constant $C:= C(\delta)> 0$ (independent of the choice of $\bar{\aleph} \in \bar{I}_2$) such that 
$
    \abs{u_{\AS}(x; \bar{\gamma})} \leq C (-x)^{1/2}
$. Thus,  there exists a  constant $\mathcal{C}_2:= \mathcal{C}_2(\delta)> 0$ such that 
\begin{align}
    \bigg | \int_{-(\frac{2\sqrt{2}}{3})^{-2/3}s^{1- \frac{2}{3}\delta}}^{-x_0}
    u_{\AS}(x; \bar{\gamma})
    \bigg | 
    =
    \bigg | \int_{\bar{\aleph} \in \bar{I}_2}
    u_{\AS}(x; \bar{\gamma})
    \bigg | 
    \leq \mc{C}_2 s^{\frac{3}{2} - \delta} \, .
    \label{eqn:mu:3}
\end{align}

Finally, consider interval $(4)$.
For any fixed $x_0$, the integral of $u_{\AS}(x;\bar{\gamma})$ over $x$ in interval $(4)$  evaluates to a constant due to the exponential decay in~\eqref{eqn:uas-bdry}. That is, there exists a positive constant $\mc{C}_3$ such that 
\begin{align}
    \abs{\int_{-x_0}^{\infty} u_{\AS}(x;\bar{\gamma})~dx} = \mc{C}_3\,.
    \label{eqn:mu:x0}
\end{align}
We are now ready to prove Lemma~\ref{1.7.mu_est}. 

\begin{proof}[Proof of Lemma~\ref{1.7.mu_est}]
Equation~\eqref{1.7.mu_est.eqn.weak} follows immediately from \eqref{eqn:mu:1}, \eqref{eqn:mu:3}, and \eqref{eqn:mu:x0}. Equation~\eqref{1.7.mu_est.eqn} follows from the additional input \eqref{eqn:mu:2stokes}.
\end{proof}

\section{Proof of Theorem~\ref{1.4.final}}
\label{section.thm.1.4.proof}
The proof of Theorem~\ref{1.4.final} was sketched in Section~\ref{subsec:intro-GOEPP}, starting from \eqref{eqn:lowerdev-markov-1}. Here, we give a complete proof.
The following corollary follows from Theorem~\ref{1.7.rewrite.F_1} and a less precise formulation of \cite[Theorem~1.4]{BB17}, which
states that $\log \mc{F}_1(-s,v)$ is given by the right-hand side of  \eqref{eqn:bb17-fixed} (and thus, by Theorem~\ref{1.7.rewrite.F_1}, the same is true for $F_1(-s,v)$).
\begin{corollary}[{\cite[Theorem~1.4]{BB17}}]
\label{cor:bb17-fixed}
Fix $\gamma \in [0,1)$ and define $v := - \log(1-\gamma) \in [0,\infty)$. There exist positive constants $S_0 := S_0(\gamma)$ such that for all $s \geq S_0$, we have 
\begin{align}
    \log F_1(-s, v) = -\frac{2}{3\pi}vs^{3/2} + \frac{v^2}{2\pi^2} \log(8 s^{3/2}) + \mc{O}(1)\,.
    \label{eqn:bb17-fixed}
\end{align}
\end{corollary}

\begin{proof}[Proof of Theorem~\ref{1.4.final}]
Fix $\eta >0$, $c>0$, and $\delta \in (0,2/5)$.
For brevity, we write $\mc{A}$ to denote the event 
\[
    \mc{A} := \left \{
    \chi^{\GOE}[-s,\infty) - \E[ \chi^{\GOE}([-s,\infty))] \leq -cs^{3/2}
    \right \}.
\]
For any  $\ld >0$, taking $f(x) = e^{-\ld x}$ in Markov's inequality gives the upper-bound
\begin{align}
    \P(\mc{A}) &\leq \exp \pth{- c \lambda s^{3/2} + \lambda \E \brak{\chi^{\GOE}([-s,\infty))}}
    \E \brak{ \exp \pth{ - \lambda \chi^{\GOE} ([-s,\infty))}} \nonumber \\
    &= \exp \pth{ -c \lambda s^{3/2} 
    + \frac{2}{3\pi} \lambda s^{3/2}  
    + \lambda D_1(s) } F_1(-s, \lambda)
    \label{1.4_anal_1.7_weakMarkov} \,, 
\end{align}
where \eqref{1.4_anal_1.7_weakMarkov} follows from the substitution of \eqref{1.3_exp}. 
Taking $\lambda = 2 \eta /c$ and substituting~\eqref{eqn:bb17-fixed} into~\eqref{1.4_anal_1.7_weakMarkov} yields
\begin{align*}
     \P(\mc{A}) &\leq \exp \pth{ -2 \eta s^{3/2} + \mc{O}(\log s) } \leq \exp \pth{ - \eta s^{3/2}}\,,
\end{align*}
where the last inequality holds for all $s$ sufficiently large (depending on $\eta$ and $c$). Thus, we have \eqref{eqn:1.4.weak}.

Now, assume Conjecture~\ref{conjecture}. Then taking $\lambda = \frac{1}{2} s^{\frac{3}{2}-\delta}$ in \eqref{1.4_anal_1.7_weakMarkov}
gives
\begin{align*}
     \P(\mc{A}) &\leq
    \exp \pth{ - \frac{1}{2}c s^{3-\delta} 
    + \frac{1}{3\pi}  s^{3-\delta}  
    + \frac{1}{2} s^{\frac{3}{2}-\delta} D_1(s)} 
    F_1 \pth{-s, \frac{1}{2} s^{\frac{3}{2}-\delta}}\,.
\end{align*}
Substituting the bound of Theorem~\ref{1.7.final} into the above yields  equation \eqref{1.4.final.eq}.
\end{proof}

\section{Proof of Theorem~\ref{1.5_anal}}
\label{section.thm.1.5.}

We now prove Theorem~\ref{1.5_anal}. Our method of proof necessarily differs from the GUE case of \cite{CG18}, which benefits from the Airy kernel being a  \textit{locally admissible} and \textit{good} \textit{trace-class operator} (see \cite[Section~4.2]{AGZ10}). 
For such kernels, on any compact set $D \subset \R$, the point process can be expressed as the following sum:
\[
\chi^{\rm{Ai}}(D) \stackrel{(d)}{=} \sum_{i=1}^{\infty} X_i,
\]
where the $X_i$ are independent Bernoulli random variables satisfying $\P(X_i= 1) = 1- \P(X_i = 0) = \ld_i^D$. Here, $\ld_i^D$ are the eigenvalues of the operator $\mathds{1}(D) K^{\rm{Ai}} \mathds{1}(D)$. An application of Bennet's concentration inequality yields the desired upper large deviations bound on $\chi^{\rm{Ai}}$. 

Pfaffian point processes possess matrix-valued kernels (see Section \ref{section.goe}), and while \cite{Kar14} describes a such class of kernels whose corresponding Pfaffian point processes can be expressed as a sum of Bernoulli random variables, no such result is known for the GOE point process. 
Instead, we estimate $\chi^{\GOE}$ on intervals by carefully analyzing the closest GOE points to the boundary of the interval. The result is the exponential upper bound \eqref{1.5.anal.thm.eqn}, which suffices to establish  \eqref{4.5.anal.eqn}, which in turn gives the lower bound \eqref{1.5.eqn.anal} on the half-space KPZ tail.

\begin{proof}[Proof of Theorem~\ref{1.5_anal}]
Throughout this proof, we write $\chi := \chi^{\GOE}$ for brevity.
Fix $c > 0$ and $\delta \in (0,2/5)$. 
In what follows, we will write $\hat{c} := \hat{c}(c)$ to denote a positive constant depending only on the parameter $c$ whose value may change from line to line.
We first consider $\fk{B}_k(\ell)$ for $k \geq 2$.

As usual, let $(\rm{a}_1> \rm{a}_2 > \dots)$
denote the GOE point process, and let $(\ld_1 < \ld_2 < \dots)$ denote the eigenvalues of the Airy operator. 
Define
\begin{align*}
        m_1 &:= \sup \{ m: \a_m \geq -(k-1)\ell \} \, ,
    \quad m_2 := \sup \{ m: \a_m \geq -k\ell \} \,, \text{ and } 
    \\
    k_1 &:= \sup\{n: -\ld_n \geq -(k-1)\ell \} \, ,
    \quad k_2 := \sup\{n: -\ld_n \geq -k\ell \} \, .
\end{align*}
Note that $\chi(\fk{B}_k(\ell)) = m_2 - m_1$.
Theorem~\ref{1.3_anal_prop} gives us
\begin{align}
    \E \brak{\chi(\fk{B}_k(\ell))}
    &= 
    \frac{2}{3\pi} (k^{3/2} - (k-1)^{3/2}) \ell^{3/2} 
    + f_1 \, ,
    \label{eqn:chi-bk-1}
\end{align}
where $f_1 := f_1(k,\ell) = D_1(k\ell) - D_1((k-1)\ell))$; note that $f_1$ is bounded in $k$ and $\ell$. By Taylor's theorem, we have
\begin{align}
    k^{3/2} - (k-1)^{3/2} = \frac{3}{2}(k-1)^{1/2} + R_k \, ,
    \label{eqn:k3/2-taylor}
\end{align}
where 
$0<R_k \leq \frac{3}{4}$.
By Corollary~\ref{1.5_anal_expchi}, we have 
\begin{align}
    \E[ \chi(\fk{B}_k(\ell)) ] = k_2 - k_1 + f_2 \,,
    \label{eqn:chi-bk-2}
\end{align}
where $f_2 := f_2(k,\ell)$ is bounded in $k$ and $\ell$. 
Define the positive constant 
\[
\fk{c}_k := \fk{c}_k(c)  = c \pth{ \frac{1}{\pi}(k-1)^{1/2} +\frac{2}{3\pi} R_k }^{-1}  \, ,
\]
which is bounded above uniformly in $k$, and satisfies
\begin{align}
    \fk{c}_k \geq \hat{c} k^{-1/2} \, .
    \label{eqn:ck-bound}
\end{align}
Then substituting \eqref{eqn:k3/2-taylor} and  \eqref{eqn:chi-bk-2} into \eqref{eqn:chi-bk-1}
yields
\[
c \ell^{3/2} = \fk{c}_k (k_2 -k_1) - f_3\, ,
\]
where $f_3 := f_3(k,\ell)$ is bounded in $k$ and $\ell$. The above display along with the relation $\chi(\fk{B}_k(\ell)) = m_2 - m_1$ gives 
\begin{align}
    \left \{ 
    \chi(\fk{B}_k(\ell)) - \E\brak{\chi(\fk{B}_k(\ell)} \geq  c \ell^{3/2} 
    \right \} 
    &= 
    \{ m_2 - m_1 \geq (1+ \fk{c}_k)(k_2 - k_1) + f_3 \} \, . \label{1.5_anal_splitthis}
\end{align}
It follows that the event $\left \{ 
    \chi(\fk{B}_k(\ell)) - \E\brak{\chi(\fk{B}_k(\ell)} \geq  c \ell^{3/2} 
    \right \}$ is contained in the event
\begin{align}
\left \{ m_2 \geq k_2 + \frac{\fk{c}_k}{2}(k_2- k_1) + f_3
\right \} 
    \cup
\left \{ m_1 \leq k_1 - \frac{\fk{c}_k}{2}(k_2 - k_1) \right \} 
    \,. 
\label{1.5_anal_termsplit}
\end{align}
The next two claims provide an upper-bound on each of the events in the above union. 
%
%
\begin{claim} \label{claim.1.5.mainterm.bd} 
There exist positive constants $\bar{c} := \bar{c}(c)$, $\kappa:= \kappa (c,\delta)$, and $\ell_0 := \ell_0(c,\delta)$ 
such that for all $\ell \geq \ell_0$, we have
\begin{align}
    \P \pth{
    m_2 \geq k_2 + \frac{\fk{c}_k}{2}(k_2- k_1)
    + f_3
    } 
    \leq  \kappa \exp \pth{
    -\kappa \pth{\bar{c}\ell}^{1-\delta}
    } \, . \label{1.5_anal_mainterm_final}
\end{align}
\end{claim}

\begin{proof} [Proof of Claim~\ref{claim.1.5.mainterm.bd}]

Since $-\a_{m_2} \leq k\ell$,  Theorem~\ref{1.6_anal} yields  
\[
(1-\ep)\ld_{m_2} - k\ell \leq C_{\ep}^{\GOE} 
\, ,
\]
 for any $\ep \in (0,1)$. 
Let $k_3 := k_2 + \frac{\fk{c}_k}{2}(k_2- k_1) + f_3$. 
Since $\ld_{i} < \ld_j$ if and only if $i < j$, the previous display gives us 
\begin{align}
    \{m_2 \geq k_3
    \} \subseteq 
    \{
    (1-\ep)\ld_{k_3} - k\ell \leq C_{\ep}^{\GOE} \} \, ,
    \label{1.5_anal_mainterm_1.6}
\end{align}
for any $\ep \in (0,1)$.
Corollary~\ref{1.5_anal_expchi} allows us to write
\begin{align}
    k_1 &= \frac{2}{3\pi}\pth{ (k-1)\ell }^{3/2} + C_1((k-1)\ell) 
    \label{1.5_anal_mainterm_k1} \, , ~\text{and} \\
    k_2 &= \frac{2}{3\pi}\pth{ k\ell }^{3/2} + C_2(k\ell) 
    \label{1.5_anal_mainterm_k2} \, , 
\end{align}
where $\sup_{x>0} \{\abs{C_1(x)}, \abs{C_2(x)} \} < 1$.
Then, from Proposition~\ref{airy_eval_bound} and the definition of $k_3$, we compute
\begin{align}
    \ld_{k_3} &= \pth{
    (k\ell)^{3/2} + \frac{\fk{c}_k}{2}\pth{ (k\ell)^{3/2} -\pth{ (k-1)\ell }^{3/2}} + f_4
    }^{2/3} 
    \nonumber \\
    &= 
    (k\ell)\pth{
    1+ \frac{\fk{c}_k}{2} \bigg (1- \Big ( \frac{k-1}{k} \Big )^{3/2}} +  (k\ell)^{-3/2} f_4
    \bigg )^{2/3} \, ,
 \label{1.5_anal_mainterm_2}
\end{align}
where $f_4 := f_4(k,\ell)$ is bounded in $k$ and $\ell$. Since the function $g(x) := x^{2/3}$ is an increasing function in $x$, \eqref{1.5_anal_mainterm_2} gives us
\begin{align}
    \ld_{k_3} \geq (k\ell) \Big ( 1+ \frac{\fk{c}_k}{4} \Big )^{2/3} \, ,
    \label{1.5_anal_mainterm_3}
\end{align}
for all $\ell \geq 1$ (and recall that we have fixed $k \geq 2$). 
Substituting \eqref{1.5_anal_mainterm_3} into \eqref{1.5_anal_mainterm_1.6}, we find 
\begin{align}
\{m_2 \geq k_3 \} \subseteq 
    \Big \{
    C_{\ep}^{\GOE} \geq  k\ell \Big ( (1-\ep) \Big ( 1+ \frac{\fk{c}_k}{4} \Big )^{2/3} - 1 \Big ) \Big \} \, . 
    \label{1.5_anal.claim.1.eqn3}
\end{align}

We now show that there exists some $\ep \in (0,1)$ such that  $k\Big ( (1-\ep) \Big ( 1+ \frac{\fk{c}_k}{4} \Big )^{2/3} - 1 \Big )$ can be bounded below by a positive constant $\bar{c} := \bar{c}(c)$ uniformly in $k \in \Z_{\geq 2}$. 
Define 
\[
\hat{\fc}_k := \hat{\fk{c}}_k(\ep) = \Big ( (1-\ep) \Big ( 1+ \frac{\fk{c}_k}{4} \Big )^{2/3} - 1 \Big ) \, .
\]
It is clear that from \eqref{eqn:ck-bound} that for any fixed $k$, there exists $\ep > 0$ such that $\hat{\fk{c}}_k > 0$. Thus, we need only consider $k$ arbitrarily large. We show that there exists a positive constant $K := K(c)$ such that for all $k \geq K(c)$, there exists $\ep := \ep(k,c) >0$ such that $\hat{\fc}_k(\ep) = k^{-1}$. 
Towards this end, using \eqref{eqn:ck-bound},  we find the lower-bound
\begin{align}
    \ep 
    &= 1- \frac{1+\hat{\fc}_k}{\Big( 1+ \frac{\fc_k}{4} \Big)^{2/3}}
    \geq 1- \frac{1+ \hat{\fc}_k}{\pth{1+ \hat{c}(k-1)^{-1/2}}^{2/3}} \, .
\end{align}
That $\ep < 1$ is trivial. 
Thus, it suffices to show that there exists a positive constant $K:= K(c)$ such that 
\begin{align}
     k^{-1} < \pth{1+ \hat{c}(k-1)^{-1/2}}^{2/3} -1 \,, \label{eqn:claim6.1-cgoal}
\end{align}
for all $k \geq K$ (for then it will follow that there exists $\ep(k,c) \in (0,1)$ such that $\hat{c}_k  =k^{-1}$, for all $k \geq K$).
Let $K:= K(c)$ be large enough such that $\hat{c}(k-1)^{-1/2}  < 1$ for all $k \geq K$. Then, by Taylor's theorem, we have 
\begin{align}
    \pth{1+ \hat{c}(k-1)^{-1/2}}^{2/3}  -1 
    = \frac{2}{3}\hat{c}(k-1)^{-1/2}+ \O(k^{-1}) 
    \geq \hat{c} (k-1)^{-1/2} \, ,
    \label{eqn:claim6.1-ctaylor}
\end{align}
where the last inequality holds for $K:= K(c)$ large enough and all $k \geq K$ (and the $\hat{c}$ on the right-most side differs from the other $\hat{c}$). Now, choose $K$ large enough such that, for all $k \geq K$, we have 
$k^{-1} < \hat{c} (k-1)^{-1/2}$. Then, from \eqref{eqn:claim6.1-ctaylor}, it follows that \eqref{eqn:claim6.1-cgoal} holds. Thus, we may take 
\[
 \bar{c} = \min \{ 1, \min_{k \leq K} k\fc_k \} \, ,
\]
which depends only on $c$.

Now, let $k_0:= k_0(c)\in \Z_{\geq 1}$ and $\ep_0:= \ep_0(c) \in (0,1)$ be such that $\bar{c} = c_{k_0}(\ep_0)$.
Thus, from \eqref{1.5_anal.claim.1.eqn3}, we have 
\begin{align}
    \{ m_2 \geq k_3 \} \subseteq \{ C_{\ep_0}^{\GOE} \geq \bar{c} \} \, .
    \label{eqn:claim6.1-barc-replace}
\end{align}
Equation \eqref{eqn:claim6.1-barc-replace} and Theorem~\ref{1.6_anal} then give the final result: there exist positive constants $\kappa:= \kappa(c,\delta)$ and $L_0 := L_0(c,\delta)$ such that for all $\ell \geq \ell_0$, we have 
\begin{align*}
    \P \pth{
    m_2 \geq k_2 + \frac{\fc_k}{2}(k_2- k_1)
    }
    &\leq  \P \pth{ C_{\ep_0}^{\GOE} \geq  \bar{c} \ell} 
    \leq  \kappa \exp \pth{
    -\kappa \pth{\bar{c} \ell}^{1-\delta}
    } \, .
\end{align*}
This concludes the proof of Claim~\ref{claim.1.5.mainterm.bd}.
\end{proof}
%
%
%

\begin{claim} \label{claim.1.5.smalltermbd}
For any $\eta >0$, there exists  a positive constant $\bar{L}_0:= \bar{L}_0 (c,\eta)$ such that for all $\ell \geq \bar{L}_0$, we have 
\begin{align}
    \P \pth{ m_1 \leq k_1 - \frac{\fc_k}{2}(k_2 - k_1) } 
    \leq \exp 
    \pth{
    -\eta \ell^{3/2} 
    } \, . \label{1.5_anal_smallterm_final}
\end{align}
\end{claim}

\begin{proof}[Proof of Claim~\ref{claim.1.5.smalltermbd}]
Fix $\eta >0$. Let the left-hand side of \eqref{1.5_anal_smallterm_final} be denoted by $\mc{P}$. 
By definition of $m_1$, we have 
$m_1 = \chi\pth{-(k-1)\ell, \infty} $. 
Corollary \ref{1.5_anal_expchi} gives the expression
\begin{align*}
    m_1 - k_1 &= \chi\pth{-(k-1)\ell, \infty} - \E \brak{ \chi \pth{-(k-1)\ell, \infty} } + g_1
    \, ,
\end{align*}
where $g_1 := g_1(k,\ell)$ is bounded in $k$ and $\ell$.
This expression allows us to write $\mc{P}$ as 
\begin{align}
\mc{P}
&= 
\P \pth{
\chi\pth{-(k-1)\ell, \infty} - \E \brak{ \chi \pth{-(k-1)\ell, \infty} } 
\leq -\frac{\fc_k}{2}(k_2 - k_1) + g_1
} \, .
\label{1.5_anal_smallterm3}
\end{align}
From equations \eqref{1.5_anal_mainterm_k1}, \eqref{1.5_anal_mainterm_k2}, and \eqref{eqn:k3/2-taylor}, we may write 
\begin{align*}
    k_2 - k_1 
    &=
    \frac{2}{3\pi} \pth{ \frac{3}{2}(k-1)^{1/2}  + R_k } \ell^{3/2} + g_2 \, ,
\end{align*}
where $g_2 := g_2(k,\ell)$ is bounded in $k$ and $\ell$. The above along with \eqref{eqn:ck-bound} yield the bound
\begin{align}
-\frac{\fc_k}{2}(k_2 - k_1) + g_1
\leq - \hat{c} k^{-1/2} \Big ( \frac{3}{2}(k-1)^{1/2} + R_k \Big ) \ell^{3/2} +g_3  
\leq - \bar{C} \ell^{3/2} \,,
\label{1.5_anal_claim.2.bound1}
\end{align}
where $g_3 := g_3(k,\ell)$ is bounded in $k$ and $\ell$ and
$\bar{C} := \bar{C}(c)$ is a positive constant; and the last inequality holds
for all $\ell \geq \bar{L}_0$, where $\bar{L}_0:=\bar{L}_0(c)$ is sufficiently large.
Substituting \eqref{1.5_anal_claim.2.bound1} into the right-hand side of 
\eqref{1.5_anal_smallterm3} yields 
\begin{align}
\mc{P}
&\leq 
\P \pth{
\chi\pth{-(k-1)\ell, \infty} - \E \brak{ \chi \pth{-(k-1)\ell, \infty} } 
\leq 
-\bar{C}\ell^{3/2}
} \, .
\label{eqn:claim6.2-here}
\end{align}
We may now apply equation~\eqref{eqn:1.4.weak} of
Theorem~\ref{1.4.final}: in the notation of this theorem, we take $c$ to be~$\bar{C}$, $s$ to be $\ell$, and $\eta$ to be the same $\eta$ here. Then there exists a positive constant $\bar{L}_0:= \bar{L}_0(c,\eta)$ such that for all $\ell \geq \bar{L}_0$, we have
$    \mc{P}
    \leq \exp \pth{ - \eta \ell^{3/2} }
$
as desired.
This concludes the proof of Claim~\ref{claim.1.5.smalltermbd}.
\end{proof}

We are now ready to conclude the proof of Theorem~\ref{1.5_anal}. Define
\[
\overline{\mc{P}}:= \P \pth{ \chi(\fk{B}_k(\ell)) - \E\brak{\chi(\fk{B}_k(\ell)} \geq  c \ell^{3/2} } \, .
\]
From \eqref{1.5_anal_termsplit}, we have 
\begin{align*}
    \overline{\mc{P}}
    &\leq
    \P \pth{
    m_2 \geq k_2 + \frac{\fc_k}{2}(k_2- k_1)
    + f_3
    } 
    + 
    \P \pth{ m_1 \leq k_1 - \frac{\fc_k}{2}(k_2 - k_1) } 
    \, .
\end{align*}
Substituting the bounds obtained in 
\eqref{1.5_anal_mainterm_final} and \eqref{1.5_anal_smallterm_final} gives
\begin{align*}
    \overline{\mc{P}}
    &\leq 
    \kappa \exp \pth{
    -\kappa (\bar{c}\ell)^{1-\delta}
    } 
    + 
    \exp \pth{ -\eta \ell^{3/2} }
    \leq \exp\pth{- \mc{C} \ell^{1-\delta} }
    \,,
\end{align*}
where the first inequality holds for any fixed $\eta >0$ and all $\ell \geq L_0$, where $L_0:=L_0(c,\delta, \eta)$ is greater than or equal to $\max \{ \ell_0, \bar{L}_0 \} $. Fixing $\eta$, the second inequality above holds for  
a (possibly larger) $L_0$ and another positive constant $\mc{C} := \mc{C}(c, \delta)$. 
This concludes the proof of the result for $k \geq 2$.

Now, if $k =1$, take $m_2$ defined as in the $k \geq 2$ case. Then \eqref{1.5_anal_splitthis} holds with $m_1 = 0$, i.e., we have
\begin{align}
    \left \{ 
    \chi(\fk{B}_k(\ell)) - \E\brak{\chi(\fk{B}_k(\ell)} \geq  c \ell^{3/2} 
    \right \} 
    &= 
    \left \{ 
    \chi(\fk{B}_k(\ell)) - \E\brak{\chi(\fk{B}_k(\ell)} \geq  \fc_k \E\brak{\chi(\fk{B}_k(\ell)} + f_3
    \right \}  \nonumber \\
    &=
    \{ m_2 \geq (1+ \fc_k)(k_2 - k_1) + f_3 \}.
\end{align}
Then \eqref{1.5_anal_mainterm_final} finishes the proof for the $k=1$ case.
\end{proof}

\section{Proof of Proposition \ref{prop.4.2.anal}} 
\label{section.proof.upper.and.lower}
In this section, we prove Proposition \ref{prop.4.2.anal}, thus completing our proof of Theorem~\ref{thm.1.1.anal}.
Here, we follow  closely the method of  \cite[Section~5]{CG18}; indeed, many of the computations done there are adapted here to our case.

Before proceeding, we recall a result describing the tail behavior of $\a_1$, which follows the GOE Tracy-Widom distribution (see \cite{TW96}). The following proposition is a much simplified version of a result of \cite{BBD08}, where the authors extract precise asymptotics up to the third order (prior, the asymptotic behavior had been known by studying the asymptotics of the solutions of the Painlev\'e II equation).

\begin{proposition}[\cite{BBD08}] \label{prop.TW.GOE.dist}
Let $\a_1$ denote the top particle in the GOE point process.
Then 
\begin{align}
    \P(\a_1 < - s ) =  \exp \pth{ 
    -\frac{1}{24}  s^3 (1 + o(1)) 
    } \, . \label{4.4.TWboundeqn}
\end{align}

\end{proposition}

\subsection{Proof of the upper bound, equations \texorpdfstring{\eqref{4.4.anal.eqn.weak}}{} and  \texorpdfstring{\eqref{4.4.anal.eqn}}{}}
\label{section.upper.bound.proof}
Recall that we defined in~\eqref{def_J} 
\[
J_s(x) : = \frac{1}{2} \log(
    1+ \exp (T^{1/3}(x+s)), 
    \ \ \text{and } \ I_s(x) := \exp(-J_s(x))
\]
We will establish an upper bound on 
$\E_{\GOE} \brak{\prod_{k=1}^{\infty} I_s(\a_k)}$
by deriving a lower bound on $\sum_{k=1}^{\infty} J_s(\a_k)$. 
To this end, we denote $D_k := (-\ld_k - \a_k)_+$, where we write $x_+ := \max \{ x , 0 \}$ for any $x \in \R$.

\begin{lemma}
\label{4.4.anal.lemma.5.2.anal}
Fix $\ep \in (0, 1/3)$. 
Define
$\theta_0 := \lfloor 2s^{3/2}/3\pi \rfloor$. There exist positive constants $S_0 := S_0(\ep)$ and $R$ such that for all $s \geq S_0$ and for all $T \geq 0$,
\begin{align}
    \sum_{k=1}^{\infty} J_s(\a_k) 
    \geq 
    \frac{1}{2} T^{1/3} \pth{ 
    \frac{4s^{5/2}}{15 \pi} (1- 8\ep) 
    - \sum_{k=1}^{\theta_0} D_k - R
    } \, . \label{lemma.5.2.anal.eq}
\end{align}

\end{lemma}

\begin{proof}
We compute 
\begin{align}
    \sum_{k=1}^{\infty} J_s(\a_k) 
    &= 
    \sum_{k=1}^{\infty} J_s\pth{
    - \ld_k - D_k + (-\ld_k - \a_k)_{-} 
    } 
    \geq \sum_{k=1}^{\infty} J_s(-\ld_k - D_k) \, ,
    \label{4.4.anal.lemma.5.2.eqn5.3}
\end{align}
where the inequality comes from the fact that $J_s(x)$ is a monotonically increasing function. We now divide the sum on the right-hand side of \eqref{4.4.anal.lemma.5.2.eqn5.3} into three ranges: $[1, \theta_1], (\theta_1, \theta_2)$, and $[\theta_2, \infty)$, where we  define
\begin{align}
\mc{K} := \sup_{n\geq 1} \{ \abs{n \mc{R}(n) } \},
~~\theta_1 := \ceil{4\mc{K}},
~~\theta_2 := \ceil{ \frac{2s^{3/2}}{3 \pi} + \frac{1}{2} } \, .
\label{eqn:def-K}
\end{align}
Here, we recall $\mc{R}(n)$ from Proposition \ref{airy_eval_bound}, and note that $\mc{K} < \infty$.
Note further that $\theta_1$ does not depend on our choice of $s$, but $\theta_2$ does, and so we can choose $s$ large enough so that $\theta_1 < \theta_2$. Thus, we take $S_0$ large enough such that for all $s \geq S_0$, we have $\theta_1 < \theta_2$.
The following two claims establish appropriate lower-bounds on the sum of $J_s(-\ld_k -D_k)$ over the first two ranges of $k$. 

\begin{claim} \label{claim.thm.1.7.pt1}
For all $s \geq 0$, 
\begin{align}
    \sum_{k=1}^{\theta_1} J_s (-\ld_k - D_k) 
    \geq \frac{1}{2}T^{1/3} \pth{
    \theta_1 
    s- \theta_1 \pth{ \frac{3\pi(4\mc{K}+1)}{2} }^{2/3}
    - \sum_{k=1}^{\theta_1} D_k
    } \, . \label{4.4.anal.lemma.5.2.claim1eqn}
\end{align}
\end{claim}

\begin{proof}[Proof of Claim~\ref{claim.thm.1.7.pt1}]
Note that for any  $a \in \R$, we have $\log (1+ \exp(a)) \geq a$. It follows that $J_s (x) \geq \frac{1}{2} T^{1/3}(s+x)$. Using this and the fact that the $\ld_k$ increase in $k$, we have 
\begin{align}
    \sum_{k=1}^{\theta_1} J_s (-\ld_k - D_k) 
    \geq 
    \frac{1}{2} T^{1/3}\sum_{k=1}^{\theta_1} s- \ld_k - D_k  
    \geq 
    \frac{1}{2} T^{1/3} \pth{ \theta_1(s- \ld_{\theta_1}) - \sum_{k=1}^{\theta_1} D_k} \, .
    \label{4.4.anal.lemma.5.2.claim1eqn1}
\end{align}
From Proposition \ref{airy_eval_bound}, 
\[
\ld_{\theta_1} \leq  \pth{ \frac{3\pi
    \pth{\theta_1 - \frac{1}{4} + \frac{\mc{K}}{\theta_1}    }
}{2} }^{2/3} \, .
\]
Since $\theta_1 - \frac{1}{4} + \frac{\mc{K}}{\theta_1} \leq 4\mc{K} + 1$, \eqref{4.4.anal.lemma.5.2.claim1eqn} follows. 
This concludes the proof of Claim~\ref{claim.thm.1.7.pt1}.
\end{proof}

\begin{claim} \label{claim.thm1.7.2}
There exists a positive constant $S_0 := S_0( \ep)$ such that for all $s \geq S_0$, 
\begin{align}
    \sum_{k= \theta_1 + 1}^{\theta_2 - 1} J_s(-\ld_k - D_k)
    \geq \frac{1}{2} T^{1/3} 
    \pth{
    \frac{4s^{5/2}}{15\pi}(1-3\ep) - (\theta_1+1)s -\sum_{k= \theta_1 + 1}^{\theta_2 - 1}  D_k
    }.
    \label{4.4.anal.lemma.5.2.claim2eqn}
\end{align}
\end{claim}

\begin{proof}[Proof of Claim~\ref{claim.thm1.7.2}]
Using similar bounds as in \eqref{4.4.anal.lemma.5.2.claim1eqn1}, along with the fact that $\ld_k \leq (3\pi k /2)^{2/3}$ for all $k > \theta_1$, we find
\begin{align}
    \sum_{k= \theta_1 + 1}^{\theta_2 - 1} J_s(-\ld_k - D_k)
    \geq \frac{1}{2} T^{1/3} \sum_{k= \theta_1 + 1}^{\theta_2 - 1} \pth{ s - \pth{\frac{3\pi k}{2}}^{2/3} - D_k} \, . 
    \label{4.4.anal.lemma.5.2.claim2eqn1}
\end{align}
We now bound the following sum with an integral, as the summands are decreasing in $k$:
\begin{align}
    \sum_{k= \theta_1 + 1}^{\theta_2 - 1} \pth{ s - \pth{\frac{3\pi k}{2}}^{2/3}}
    &\geq \int_{\theta_1+1}^{\theta_2-1} s - \pth{\frac{3\pi z}{2}}^{2/3} ~dz 
    \nonumber \\
    &\geq \int_0^{\theta_2-1} s - \pth{\frac{3\pi z}{2}}^{2/3} ~dz - (\theta_1+1)s \nonumber \\
    &=(\theta_2 -1) \pth{ 
    s - \frac{3}{5}\pth{\frac{3\pi}{2}}^{2/3} (\theta_2-1)^{2/3} 
    } - (\theta_1+1)s \, .
    \label{4.4.anal.lemma.5.2.claim2eqn2}
\end{align}
Note that $\theta_2 -1 \geq \frac{2s^{3/2}}{3\pi}- \frac{1}{2}$, and thus for $s \geq \pth{\frac{3\pi}{4\ep}}^{2/3}$, we have
\[
(1- \ep) \frac{2s^{3/2}}{3\pi} \leq \theta_2 -1 \leq \frac{2s^{3/2}}{3\pi}+1.
\] 
Substituting this bound into \eqref{4.4.anal.lemma.5.2.claim2eqn2} and then substituting into \eqref{4.4.anal.lemma.5.2.claim2eqn1} leads to \eqref{4.4.anal.lemma.5.2.claim2eqn}. 
This concludes the proof of Claim~\ref{claim.thm1.7.2}.
\end{proof}

Returning to the proof of Lemma~\ref{4.4.anal.lemma.5.2.anal}, we substitute the bounds given by 
\eqref{4.4.anal.lemma.5.2.claim1eqn}, \eqref{4.4.anal.lemma.5.2.claim2eqn}, and $\sum_{k= \theta_2}^{\infty} J_s(-\ld_k-D_k) \geq 0$ into 
\eqref{4.4.anal.lemma.5.2.eqn5.3} to obtain
\begin{align}
\sum_{k=1}^{\infty} J_s(\a_k)
\geq 
\frac{1}{2}T^{1/3} 
\brak{
    \frac{4s^{5/2}}{15\pi}(1-3\ep) 
    - \theta_1 \pth{ \frac{3\pi(4\mc{K}+1)}{2} }^{2/3}
    - s -\sum_{k= 1}^{\theta_2 - 1}  D_k
}. \label{lemma.5.2.anal.lasteqn}
\end{align}
Recalling $\theta_1 :=\ceil{4\mc{K}}$, we note that $\theta_1 \pth{ 3\pi(4\mc{K}+1)/2 }^{2/3}$ is a constant which can be replaced by a large constant $R > 0$. Finally, for sufficiently large $s \geq S_0$, we have $s \leq \frac{4 \ep s^{5/2}}{3 \pi}$,
and thus we may make this replacement in \eqref{lemma.5.2.anal.lasteqn} to obtain \eqref{lemma.5.2.anal.eq}. 
This completes the proof of Lemma~\ref{4.4.anal.lemma.5.2.anal}.
\end{proof}

\begin{proof}[Proof of \eqref{4.4.anal.eqn.weak} and \eqref{4.4.anal.eqn} in Proposition \ref{prop.4.2.anal}]
From \eqref{lemma.5.2.anal.eq}, we have
\begin{align}
    \prod_{k=1}^{\infty} I_s(\a_k) 
    &= 
    \exp \pth{ - \sum_{k=1}^{\infty} J_s(\a_k) }
    \leq \exp \pth{
    - \frac{1}{2} T^{1/3} \pth{ 
    \frac{4s^{5/2}}{15 \pi} (1- 8\ep) 
    - \sum_{k=1}^{\theta_0} D_k - R
    }
    } \, ,
    \label{eqn:prodI-ub}
\end{align}
for all $s\geq S_0$ and for all $T \geq 0$.
Note that for $S_0$ sufficiently large, we have
\begin{align}
\ep s \theta_0 + R 
\leq 
\frac{4s^{5/2}}{15 \pi} \pth{ \frac{5}{2}\ep + \frac{15 \pi R}{4s^{5/2}} }
< \frac{4s^{5/2}}{15 \pi} (3 \ep) 
\label{eqn:epstheta-bd}
\end{align}
for all $s \geq S_0$.
Define $\mc{S}_{\theta_0} := \sum_{k=1}^{\theta_0} D_k$.
Then \eqref{eqn:prodI-ub} and \eqref{eqn:epstheta-bd} yield
\begin{align}
    \mathds{1} \pth{ \mc{S}_{\theta_0} < \ep s \theta_0 }
    \prod_{k=1}^{\infty} I_s(\a_k)
    &\leq \exp \pth{
    -T^{1/3} \frac{2s^{5/2}}{15 \pi} (1- 11\ep)
    } \,.
\end{align}
On the other hand, if 
$S_{\theta_0} \geq \ep s \theta_0$,
then there exists at least one $k \in [1, \theta_0] \cap \Z$ such that $D_k > \ep s$. Thus,
$\{ S_{\theta_0} \geq \ep s \theta_0 \} \subset
\bigcup_{k=1}^{\theta_0} \{ D_k \geq \ep s \} $.
It follows that 
\begin{align}
    \E_{\GOE} \brak{ \prod_{k=1}^{\infty} I_s(\a_k) }
    &= 
    \E \brak{ \mathds{1}\pth{ \mc{S}_{\theta_0} < \ep s \theta_0 } \prod_{k=1}^{\infty} I_s(\a_k) }
    + \E \brak{ \mathds{1} \pth{ S_{\theta_0} \geq \ep s \theta_0 }
    \prod_{k=1}^{\infty} I_s(\a_k)  
    } \nonumber \\
    &\leq 
    \exp \pth{
    -T^{1/3} \frac{2s^{5/2}}{15 \pi} (1- 11\ep)
    }
    + 
    \E \brak{ 
    \mathds{1} \pth{ \bigcup_{k=1}^{\theta_0} \{ D_k \geq \ep s \} }  \prod_{k=1}^{\infty} I_s(\a_k) 
    }\,. \label{4.4.anal.eqnsub}
\end{align}
We split the indicator function as
\begin{align}
    \mathds{1} \pth{\bigcup_{k=1}^{\theta_0} \{ D_k \geq \ep s \} }
    \leq \mathds{1} \pth{ \bigcup_{k=1}^{\theta_0} \{ D_k \geq \ep s \} 
    \cap \{ \a_1 \geq -(1-\ep) s \}
    }
    + \mathds{1} \pth{ \a_1 \leq - (1-\ep) s }\,.
    \label{4.4.anal.eqn1}
\end{align}
Since $I_s(\a_k) \leq 1$ for all $k \in \Z_{\geq 1}$, we have that when $\a_1 \geq -(1-\ep) s$,
\begin{align}
    \prod_{k=1}^{\infty} I_s(\a_k) 
    \leq I_s(\a_1) 
    \leq \frac{1}{\sqrt{1+ \exp\pth{T^{1/3}(s+ \a_1)}}} 
    \leq 
    \exp\pth{ -\frac{1}{2}\ep sT^{1/3} }\,. \label{4.4.anal.eqn2}
\end{align}
Substituting \eqref{4.4.anal.eqn1} and \eqref{4.4.anal.eqn2} into \eqref{4.4.anal.eqnsub} gives 
\begin{align}
    \E_{\GOE} \brak{ \prod_{k=1}^{\infty} I_s(\a_k) }
    &\leq 
    \exp \pth{
    -\frac{2(1- 11\ep)}{15 \pi} T^{1/3} s^{5/2}
    }
    + \exp\pth{ -\frac{1}{2}\ep sT^{1/3} } \P \pth{
    \bigcup_{k=1}^{\theta_0} \{D_k \geq \ep s \}
    } 
    \nonumber \\
    &+ \P (\a_1 \leq -(1-\ep)s) \,.
    \label{eqn:main-ub:laststep}
\end{align}
Using \eqref{4.4.TWboundeqn}, we have 
\begin{align}
    \P (\a_1 \leq -(1-\ep)s) =
    \exp \pth{ -(1-\ep)^3 \frac{s^3}{24} \pth{1 + o(1)} }
    \leq \exp\pth{-\frac{s^3}{24}(1-C\ep)} \,,
    \label{eqn:main-ub:TW}
\end{align}
for some constant $C > 0$ and all $s$ sufficiently large.
Now, taking $C= \max \{C, 11\}$ and using Lemma~\ref{lemma.5.3.anal}, we obtain both \eqref{4.4.anal.eqn.weak} and \eqref{4.4.anal.eqn}.
\end{proof}

\begin{lemma}
\label{lemma.5.3.anal}
Fix $\eta >0$, $\ep \in (0, 1/3)$, and $\delta \in (0, 1/4)$. Then there exist  positive constants $S_0 := S_0(\eta, \ep, \delta) >0$ and $K_1 := K_1(\ep, \delta) >0$ such that the following holds for all $s \geq S_0$.
Divide the interval $[-s, 0]$ into $\ceil{2 \ep^{-1}} +1$ segments
$\mc{Q}_i := [-j \ep s /2, -(j-1)\ep s /2)$ 
for 
$j = 1, \dots, \ceil{2 \ep^{-1}} +1$.
Denote the left and right endpoints of $\mc{Q}_j$ by $p_j$ and $q_j$ respectively.
Define
$k_j := \# \{ k: - \ld_k \geq q_j \}$, where $(\ld_1 < \ld_2< \dots)$ denote the Airy operator eigenvalues. 
Then (recalling 
$\theta_0 = \lfloor 2s^{3/2}/3\pi \rfloor $),
for all $j \in \{1, \dots, \ceil{2 \ep^{-1}}+1 \}$, 
we have
\begin{align}
    \P ( \a_{k_j} \leq p_j ) &\leq \exp \pth{-\eta s^{3/2}} \, , \text{ and}
    \label{5.3.eqn.akj.weak}\\
    \P \pth{ \bigcup_{k=1}^{\theta_0} \{D_k \geq \ep s \} } &\leq \exp \pth{-\eta s^{3/2}} \, , \label{5.3.eqn.union.weak}
\end{align}
and, assuming Conjecture~\ref{conjecture},  we have 
\begin{align}
    \P ( \a_{k_j} \leq p_j ) &\leq \exp \pth{-K_1 s^{3-\delta}} \, , \text{ and}
    \label{5.3.eqn.akj}\\
    \P \pth{ \bigcup_{k=1}^{\theta_0} \{D_k \geq \ep s \} } &\leq \exp \pth{-K_1 s^{3- \delta}} \, . \label{5.3.eqn.union}
\end{align}
\begin{proof}
 
If $\a_{k_j} \leq p_j$, then
\begin{align}
    \chi^{\GOE} \pth{ [-j\ep s/2, \infty) }
    \leq k_j \, .
    \label{1.5.anal.lemma.5.3.1}
\end{align}
Corollary \ref{1.5_anal_expchi} gives us the following expressions:
\begin{align}
    k_j &= \frac{2}{3\pi}\pth{ j\ep s/2 }^{3/2} + C_1 \pth{ j\ep s/2 } \, , \text{ and }\\
    \E \brak{\chi^{\GOE}\pth{ [-j\ep s/2, \infty) }} 
    &= \frac{2}{3\pi}\pth{ j\ep s/2 }^{3/2} + C_2 \pth{ j\ep s/2 } \, ,
    \label{eqn:j-ep-s-bd}
\end{align}
where $M' := \sup_{x \geq 0} \{ \abs{C_1(x)}, \abs{C_2(x)} \} < \infty$. 
It follows from \eqref{1.5.anal.lemma.5.3.1}--\eqref{eqn:j-ep-s-bd} that if $\a_{k_j} \leq p_j$, then 
\begin{align}
    \chi^{\GOE} &\pth{ [j\ep s/2, \infty) }
    - 
     \E \brak{\chi^{\GOE}\pth{ [-j\ep s/2, \infty) }} 
     \nonumber \\
    &\leq 
    k_j - \frac{2}{3\pi}\pth{ j\ep s/2 }^{3/2}
    - C_2 \pth{ j\ep s/2 } 
    \nonumber \\
    &=
    \frac{(\ep s)^{3/2}}{3\pi \sqrt{2}}
    \pth{(j-1)^{3/2} - j^{3/2}} 
    + C_1 \pth{ (j-1)\ep s/2 } - C_2 \pth{ j\ep s/2 } 
    \nonumber \\
    &\leq 
    -M \sqrt{j} (\ep s)^{3/2} + M'\, ,
\end{align}
where $M >0 $ is a constant extracted from the fact that 
\[
(j-1)^{3/2} - j^{3/2} \leq \sqrt{j}( (j-1)- j) = -\sqrt{j} \,.
\]
It follows that 
\begin{align*}
    \P(\a_{k_j} \leq p_j) 
    \leq
    \P \pth{ 
    \chi^{\GOE}\pth{ [p_j, \infty) }
    - 
     \E \brak{\chi^{\GOE}\pth{ [p_j, \infty) }}
     \leq
    -M \sqrt{j} (\ep s)^{3/2} + M'
    } \, .
\end{align*}
Now, for sufficiently large $S_0$, we have 
\begin{align*}
    -M \sqrt{j} (\ep s)^{3/2} + M'
    \leq -\frac{M}{2}\sqrt{j}(\ep s)^{3/2} 
\end{align*}
for all $j \in \{ 1, \dots, \ceil{2 \ep^{-1}}+1\}$ and for all $ s \geq S_0$.
Assuming Conjecture~\ref{conjecture}, we may now apply equation~\eqref{1.4.final.eq} of Theorem~\ref{1.4.final}:  there exist $S_0(\ep, \delta)$ and $K_1 = K_1(\ep, \delta)$ such that for all $s \geq S_0$, 
\begin{align}
    \P(\a_{k_j} \leq p_j) 
    \leq 
    \P \pth{ 
    \chi^{\GOE}\pth{ [p_j, \infty) }
    - 
     \E \brak{\chi^{\GOE}\pth{ [p_j, \infty) }}
     \leq
    -\frac{M}{2}\sqrt{j}(\ep s)^{3/2}
    }
    \leq \exp \pth{ K_1 s^{3-\delta} } \,.
    \label{eqn:use-lem7.5-1}  
\end{align}
This proves \eqref{5.3.eqn.akj}.
Applying~\eqref{eqn:1.4.weak} instead of~\eqref{1.4.final.eq} above yields \eqref{5.3.eqn.akj.weak} (for all $s \geq S_0$, for some $S_0 := S_0(\eta, \ep, \delta))$.

Towards showing~\eqref{5.3.eqn.union.weak} and~\eqref{5.3.eqn.union}, assume $s$ is large enough so that 
$
    \ld_{\theta_0} < s
$.
We will now show that 
\begin{align}
    \bigcup_{k=1}^{\theta_0} \{ D_k \geq \ep s \} \subset 
    \bigcup_{j=1}^{\ceil{2 \ep^{-1}}+1} \{\a_{k_j} \leq p_j \} \, . \label{lemma.5.3.anal.eqn5.13}
\end{align}
First, choose $1 \leq k \leq \theta_0$ and assume that $D_k \geq \ep s$. 
There exists $1 \leq j \leq \ceil{2 \ep^{-1}}+1$ such that $-\ld_k \in \mathcal{Q}_{j-1}$. The left boundary point of $\mc{Q}_{j-1}$ is $q_j$, and since
$D_k = -\ld_k - \a_k  \geq \ep s$, we have 
$\a_k \leq -\ld_k - \ep s$.
Since $-\ld_k \geq q_j$, by definition of $k_j$, we have
$k_j \geq k$. Thus, $\a_k \geq \a_{k_j}$. It follows that 
\[
\a_{k_j} \leq \a_k
\leq -\ld_k - \ep s  \leq -\ld_{k_j} - \frac{\ep s}{2} \, ,
\]
where the last inequality uses the fact that $\ld_{k_j}, \ld_k \in \mc{Q}_{j-1}$, and thus $0 \leq \ld_{k_j} - \ld_k \leq \ep s/2$. Hence,
the distance between $\a_{k_j}$ and $-\ld_{k_j}$ is greater than or equal to $\ep s / 2$, from which it follows that $\a_{k_j} \leq p_j$. This establishes \eqref{lemma.5.3.anal.eqn5.13}.

Assuming Conjecture~\ref{conjecture}, we may combine  \eqref{5.3.eqn.akj} and  \eqref{lemma.5.3.anal.eqn5.13} to obtain
\begin{align}
    \P \pth{\bigcup_{k=1}^{\theta_0} \{ D_k \geq \ep s \} } 
    \leq \sum_{i=1}^{\ceil{2\ep^{-1}}+1} \P \pth{\a_{k_i}\leq p_i} 
    \leq \pth{\ceil{2 \ep^{-1}} +1 } \exp \pth{ -K_1 s^{3-\delta}} \,.
\end{align}
For $S_0:= S_0(\ep, \delta)$ sufficiently large, we can modify the constant $K_1:= K_1(\ep, \delta)$ to absorb the constant $\ceil{2 \ep^{-1}} +1 $. This establishes \eqref{5.3.eqn.union}.
On the other-hand, from \eqref{5.3.eqn.akj.weak} and   \eqref{lemma.5.3.anal.eqn5.13}, we obtain
\begin{align}
    \P \pth{\bigcup_{k=1}^{\theta_0} \{ D_k \geq \ep s \} } 
    \leq \sum_{i=1}^{\ceil{2\ep^{-1}}+1} \P \pth{\a_{k_i}\leq p_i} 
    \leq \pth{\ceil{2 \ep^{-1}} +1 } \exp \pth{ -\eta' s^{3/2}}\,,
\end{align}
for any $\eta' >0$. For any given $\eta >0$, we may choose $\eta'$ sufficiently close to $0$ and $S_0:= S_0(\eta, \ep, \delta)$ sufficiently large such that 
\[
\pth{\ceil{2 \ep^{-1}} +1 } \exp \pth{ -\eta' s^{3/2}} \leq \exp \pth{ -\eta s^{3/2}}\,.
\]
Thus, we have \eqref{5.3.eqn.union.weak}. This
completes the proof of Lemma~\ref{lemma.5.3.anal}.
\end{proof}
\end{lemma}

\subsection{Proof of the lower bound, equation \texorpdfstring{\eqref{4.5.anal.eqn}}{}}
\label{section.lower.bound.proof}
In this section we establish a lower bound on
$\E [ \prod_{k=1}^{\infty} I_s(\a_k)]$ by deriving an upper bound on $\sum_{k=1}^\infty J_s(\a_k)$. 
The result will lead us to
\eqref{4.5.anal.eqn} of Proposition \ref{prop.4.2.anal},
thus completing the proof of Theorem~\ref{thm.1.1.anal}.
We begin with an algebraic inequality from \cite{CG18}.

\begin{lemma}[{\cite[Lemma~5.6]{CG18}}]
\label{lem:CG20-lem5.6}
For all $a> 27$ and all $x \geq \sqrt{3a}$, we have 
\begin{align}
    (a+x)^{2/3} \geq a^{2/3} + x^{1/3} \, .
    \label{lemma.5.6.cg18.eqn}
\end{align}

\end{lemma}

The following lemma gives the needed upper-bound on $\sum_{k=1}^{\infty} J_s(\a_k)$ when $\a_1 \geq -s$ (see Claim \ref{claim.eqn.5.25}). 
\begin{lemma} \label{lemma.5.4.anal}
Fix $T_0 >0$. There exist positive constants $S_0$ and $B := B(T_0)$ such that for all $\ep \in (0, 1/3)$, for all $s \geq S_0$, and for all $T > T_0$, we have
\begin{align}
    \sum_{k=1}^{\infty} J_s(\a_k) \leq \frac{1}{2}\mc{L}_{T,\ep} (s+ C_{\ep}^{\GOE}) \, ,
    \label{5.14.anal.eqn}
\end{align}
where
\begin{align*}
    \mc{L}_{T,\ep} (x)
    :=
    T^{1/3}
    \pth{
    \frac{4x^{5/2}}{15 \pi}(1+ 3\ep) + 2x - B
    } 
    + 
    \frac{x^{3/2}}{3(1-\ep)^{3/2}}
    + 
    \sqrt{\frac{3}{\pi}}
    \frac{x^{3/4}}{(1-\ep)^{3/4}}
    + 
    \frac{4}{T \pi (1-\ep)^3} \, .
\end{align*}

\begin{proof}
Recall from \eqref{def_J} that $J_s(x)$ is a monotonically increasing function, and recall from \eqref{1.6.anal.eqn} that $\a_k \leq -(1-\ep) \ld_k + C_{\ep}^{\GOE}$, for all $k \in \Z_{>0}$. 
It follows that 
\begin{align}
    \sum_{k=1}^{\infty} J_s( \a_k) 
    \leq 
    \sum_{k=1}^{\infty} J_s \pth{-(1-\ep) \ld_k + C_{\ep}^{\GOE}} 
    = (\widetilde{I}) + (\widetilde{II}) + (\widetilde{III}) \, ,
    \label{5.15.eqn.anal}
\end{align}
where $(\widetilde{I})$, $(\widetilde{II})$, and $(\widetilde{III})$ equal the sum of $J_s \pth{-(1-\ep) \ld_k + C_{\ep}^{\GOE}}$ over all integers $k$ in the intervals 
$[1, \theta_1'], (\theta_1', \theta_2')$, and $[\theta_2', \infty)$ respectively, and we define 
\begin{align*}
    \theta_1' &:= \ceil{ 4 \sup_{n \in \Z_{>0}} n \abs{\mc{R}(n)} } \, , \text{ and} \\
    \theta_2' &:= \ceil{\frac{2(s + C_{\ep}^{\GOE})^{3/2}}{3\pi(1-\ep)^{3/2}} + \frac{1}{2}} \, ,
\end{align*}
where $\mc{R}(n)$ is defined as in Proposition \ref{airy_eval_bound}.
Since the $\ld_i$ are strictly decreasing in $i$,
we have 
\[
J_s \pth{-(1-\ep) \ld_k + C_{\ep}^{\GOE}} 
\leq J_s \pth{-(1-\ep) \ld_1 + C_{\ep}^{\GOE}} \, ,
\]
for all $k \geq 1$. 
Using this and the inequality 
$\log ( 1 + \exp(a) ) \leq a + \pi/2$ for any $a> 0$, 
we obtain
\begin{align}
    (\widetilde{I}) \leq 
    \theta_1' J_s \pth{-(1-\ep) \ld_1 + C_{\ep}^{\GOE}}
    \leq 
    \frac{1}{2} \pth{
    \theta_1' T^{1/3} \pth{s- (1- \ep) \ld_1 + C_{\ep}^{\GOE}}
    + \frac{\pi \theta_1'}{2}
    }. \label{5.16.anal.eqn}
\end{align}
Terms $(\widetilde{II})$ and $(\widetilde{III})$ are bounded in the following two claims.

\begin{claim} \label{claim.eqn.5.17}
For all $s > 0$, we have
\begin{equation}
    2(\widetilde{II}) \leq 
    T^{1/3} \pth{
    \frac{4(s+C_{\ep}^{\GOE})^{5/2}}{15 \pi} (1+ 3\ep) 
    +(2- \theta_1')(s + C_{\ep}^{\GOE}) 
    - 
    \frac{3}{5} \pth{\frac{3\pi}{2}}^{2/3}
    \pth{\theta_1'}^{5/3}
    } 
    + \frac{\pi(\theta_2' - \theta_1')}{2}
    \, .
    \label{5.17.anal.eqn}
\end{equation}

\begin{proof}[Proof of Claim~\ref{claim.eqn.5.17}]
Recall the constant $\mc{K}$, defined in \eqref{eqn:def-K}. 
It follows that for $k \in (\theta_1', \infty)$, we have
\[
\abs{\mc{R}(k)} \leq \frac{\mc{K}}{k} \leq \frac{\mc{K}}{\theta_1'} 
\leq 1/4 \, .
\]
Combining this with Proposition \ref{airy_eval_bound}, we find 
\begin{align}
    \ld_k \geq \pth{  
    \frac{3\pi   \pth{ k- \frac{1}{4} - \abs{\mc{R}(k)} }}{2}
    }^{2/3} 
    \geq \pth{\frac{3\pi(k- \frac{1}{2})}{2}}^{2/3} \, .
    \label{5.18.eqn.anal}
\end{align}
Using this, the inequality 
$\log ( 1 + \exp(a) ) \leq a + \pi/2$
for any $a>0$, and the monotonicity of $J_s(\cdot)$,  we obtain
\begin{align}
    (\widetilde{II}) \leq
    \frac{1}{2} \sum_{k= \theta_1' +1}^{\theta_2'-1} \pth{T^{1/3} f_s(k) + \frac{\pi}{2}} \, ,
    \label{5.19.anal.eqn}
\end{align}
where
\[
f_s(z) := s+ C_{\ep}^{\GOE} - (1- \ep) \pth{\frac{3\pi(z- \frac{1}{2})}{2}}^{2/3} \, .
\]
Since $f_s(z)$ is a monotonically decreasing function of $z$, we may bound the sum in \eqref{5.19.anal.eqn} with an integral:
\begin{align}
    \frac{1}{2}
    \sum_{k= \theta_1' +1}^{\theta_2'-1} \pth{T^{1/3} f_s(k) + \frac{\pi}{2}}
    & \leq 
    \frac{1}{2} \pth{ T^{1/3} \int_{\theta_1'}^{\theta_2'} f_s(z) ~dz
    + \frac{\pi (\theta_2' - \theta_1')}{2}
    } \, .
    \label{5.20.anal.eqn}
\end{align}
We now compute
\begin{align}
    \int_{\frac{1}{2}}^{\theta_2'} f_s(z) ~dz
    &=
    (s + C_{\ep}^{\GOE}) \pth{ \theta_2'- \frac{1}{2} } 
    - \frac{3(1-\ep)}{5}  \pth{\frac{3\pi}{2}}^{2/3}
    \pth{\theta_2' - \frac{1}{2}}^{5/3} \nonumber \\ 
    &\leq 
    (s + C_{\ep}^{\GOE}) \pth{ \frac{2(s+C_{\ep}^{\GOE})^{3/2}}{3\pi (1-\ep)^{3/2}} + \frac{3}{2} }
    - \frac{3(1- \ep)}{5} 
    \pth{\frac{3\pi}{2}}^{2/3}\pth{ \frac{2(s+C_{\ep}^{\GOE})^{3/2}}{3\pi (1-\ep)^{3/2}} }^{5/3} \nonumber \\ 
    &= \frac{4(s+ C_{\ep}^{\GOE})^{5/2}}{15(1-\ep)^{3/2}}
    + \frac{3}{2} \pth{s + C_{\ep}^{\GOE}} 
    \nonumber \\ 
    &\leq \frac{4(s+ C_{\ep}^{\GOE})^{5/2}}{15}(1+3\ep)
    + \frac{3}{2} \pth{s + C_{\ep}^{\GOE}} \, ,
    \label{5.17.anal.theta2bound}
\end{align}
and
\begin{align}
    \int_{\frac{1}{2}}^{\theta_1'} f_s(z) ~dz 
    &\geq 
    (s + C_{\ep}^{\GOE})\pth{ \theta_1' - \frac{1}{2} } 
    - \int_{\frac{1}{2}}^{\theta_1'} \pth{ 
    \frac{3\pi \pth{z- \frac{1}{2}}}
    {2}
    }^{2/3} ~dz 
    \nonumber \\ 
    &= (s + C_{\ep}^{\GOE}) \pth{\theta_1' - \frac{1}{2}} - \frac{3}{5} \pth{\frac{3\pi}{2}}^{2/3}
    \pth{\theta_1'}^{5/3} \, .
    \label{5.17.anal.theta1bound}
\end{align}
Substituting the bounds from \eqref{5.17.anal.theta2bound} and \eqref{5.17.anal.theta1bound} into 
\eqref{5.20.anal.eqn} yields the upper bound on $(\widetilde{II})$ in \eqref{5.17.anal.eqn}.
This completes the proof of Claim~\ref{claim.eqn.5.17}.
\end{proof}
\end{claim}

\begin{claim} \label{claim.eqn.5.21}
There exists a positive constant $S_0 > 0$ such that for all $s \geq S_0$, we have 
\begin{align}
    (\widetilde{III}) 
    \leq 
    \frac{1}{2} \pth{
    \sqrt{\frac{3}{\pi}} 
    \frac{\pth{s+ C_{\ep}^{\GOE}}^{3/4}}
    {(1-\ep)^{3/4}}
    + \frac{4}{T \pi (1-\ep)^3}
    } \, . 
    \label{5.21.eqn}
\end{align}

\begin{proof}[Proof of Claim~\ref{claim.eqn.5.21}]
Using the inequality $\log(1+z) \leq z$ for all $z \geq 0$, we obtain
\begin{align}
    J_s \pth{ 
    -(1-\ep)\ld_k + C_{\ep}^{\GOE} 
    } 
    \leq 
    \frac{1}{2}\exp \pth{ 
    T^{1/3} \pth{
    s- (1-\ep)\ld_k + C_{\ep}^{\GOE}
    }
    }. \label{5.22.eqn.anal}
\end{align}
Recalling the lower bound on $\ld_k$ from \eqref{5.18.eqn.anal} and the definition of $f_s(z)$ from \eqref{5.19.anal.eqn}, we find
\begin{align}
    (\widetilde{III}) \leq
    \frac{1}{2}
    \sum_{k= \theta_2'}^{\infty} \exp \pth{ 
    T^{1/3} f_s(k)
    } \, . \label{5.23.eqn.anal}
\end{align}
For all $k \geq \theta_2'$, we have 
\[
s+ C_{\ep}^{\GOE} < (1- \ep) \pth{
\frac{3\pi (\theta_2' - \frac{1}{2})}{2}
}^{2/3} \, .
\]
Since $f_s(z)$ is a monotonically decreasing function, we have 
$f_s(k) \leq f_s(\theta_2') < 0$
for all $k \geq \theta_2'$.
Thus, for all $k > \theta_2' + \sqrt{3 \theta_2'}$, $S_0$ sufficiently large, and for all $s \geq S_0$, we may write
\begin{align}
    f_s(k) <
    (1- \ep) \pth{ \pth{
\frac{3\pi (\theta_2' - \frac{1}{2})}{2}
}^{2/3}
    - 
    \pth{
\frac{3\pi (k - \frac{1}{2})}{2}
}^{2/3}
    }
    \leq
    -(1- \ep) \pth{ 
    \frac{3\pi (k- \theta_2')}{2}
    }^{1/3} \, , \label{lemma.5.5.anal.claim2.eqn1}
\end{align}
where the last inequality uses \eqref{lemma.5.6.cg18.eqn} with
\[
a := \frac{3\pi}{2} \pth{ \theta_2' - \frac{1}{2}}, \ \ \ \ \ \ 
x := \frac{3\pi}{2}( k - \theta_2 ') \, 
\]
($S_0$ need only be large enough so that $a$ and $x$ as above satisfy the conditions of Lemma~\ref{lem:CG20-lem5.6} for all $s \geq S_0$).
It follows from \eqref{lemma.5.5.anal.claim2.eqn1} and $f_s(k) <0$ that
\begin{align}
    \exp\pth{T^{1/3} f_s(k)}
    \leq 
    \begin{cases}
    1, &\text{for } k \in \left [ \theta_2', \theta_2' + \sqrt{3\theta_2'} \right ) \\
    \exp \pth{ 
    -(1- \ep) \pth{ 
    \frac{3\pi (k- \theta_2')}{2}
    }^{1/3}}, 
    &\text{for } k \in \left [ \theta_2'+ \sqrt{3\theta'}, \infty \right )
    \end{cases} \,,
\end{align}
for $S_0$ sufficiently large and for all $s \geq S_0$. 
From \eqref{5.23.eqn.anal} and the above, we find that for $S_0$ sufficiently large and all $s \geq S_0$,
\begin{align}
    2(\widetilde{III}) 
    &\leq 
    \sum_{k\in  \left [ \theta_2', \theta_2' + \sqrt{3\theta_2'} \right )} \exp \pth{ T^{1/3} f_s(k) }
    + \sum_{k \geq \theta_2' + \sqrt{3\theta_2'}} \exp \pth{ T^{1/3} f_s(k) } 
    \nonumber \\ 
    &\leq 
    1+ \sqrt{3\theta_2'} 
    + \sum_{k= \theta_2'+ \sqrt{3\theta'}}^{\infty} \exp \pth{ 
    -(1- \ep) \pth{ 
    \frac{3\pi (k- \theta_2')}{2}
    }^{1/3}} 
    \nonumber \\ 
    &\leq 
    1+\sqrt{3 \theta_2'} + \int_0^{\infty} \exp \pth{ 
    - (1- \ep) T^{1/3} \pth{ 
    \frac{3\pi z}{2}}^{1/3}
    } ~dz 
    \nonumber \\ 
    &= 1+\sqrt{3\theta_2'} + \frac{4}{T \pi(1-\ep)^3}  
    \nonumber \\
    &\leq \sqrt{ \frac{3}{\pi}} \frac{(s+ C_{\ep}^{\GOE})^{3/4}}{(1- \ep)^{3/4}} + \frac{4}{T \pi(1-\ep)^3} \,.
\end{align}
This completes the proof of \eqref{5.21.eqn} of Claim~\ref{claim.eqn.5.21}.
\end{proof}
\end{claim}

We now return to the proof of Lemma~\ref{lemma.5.4.anal}.
Define the bounded, positive constant
\[
B':=  
    \frac{3}{5} \pth{\frac{3\pi}{2}}^{2/3}     \pth{\theta_1'}^{5/3} + (1- \ep) \theta_1' \ld_1 \, .
\]
Then substituting the bounds given by
\eqref{5.16.anal.eqn}, 
\eqref{5.17.anal.eqn}, and 
\eqref{5.21.eqn} 
into \eqref{5.15.eqn.anal} yields
\begin{align}
    2\sum_{k=1}^{\infty} J_s( \a_k) 
    &\leq 
    T^{1/3} \pth{
    \frac{4(s+C_{\ep}^{\GOE})^{5/2}}{15 \pi} (1+ 3\ep) 
    +2(s + C_{\ep}^{\GOE})  
    - B'
    }
    + \frac{\pi \theta_2'}{2} \\
    &+ \sqrt{\frac{3}{\pi}} 
    \frac{\pth{s+ C_{\ep}^{\GOE}}^{3/4}}
    {(1-\ep)^{3/4}}
    + \frac{4}{T \pi (1-\ep)^3} \,. 
\end{align}
Now,
\begin{align}
    \frac{\pi \theta_2'}{2} 
    \leq \frac{\pi }{2} \pth{ 
    \frac{2}{3\pi} \frac{(s + C_{\ep}^{\GOE})^{3/2}}{(1-\ep)^{3/2}} + \frac{3}{2}
    } 
    = \frac{(s + C_{\ep}^{\GOE})^{3/2}}{3(1-\ep)^{3/2}}+ \frac{3\pi}{4} \,.
\end{align}
Taking $B : = B' - \frac{3\pi}{4T_0^{1/3}}$ yields \eqref{5.14.anal.eqn}.
\end{proof}
\end{lemma}

\begin{proof}[Proof of \eqref{4.5.anal.eqn} of Proposition \ref{prop.4.2.anal}]
In what follows, 
we fix $\ep \in (0,1/3)$, $\delta \in (0,1/4)$, and $T_0 > 0$.
We begin with two claims.

\begin{claim}
\label{claim.eqn.5.25}
There exist $\kappa := \kappa ( \ep, \delta) >0$ and $S_0 = S_0 ( \ep, \delta, T_0) >0$ such that,
for all $s \geq S_0$ and $T > T_0$,
\begin{align}
    \E_{\GOE} \brak{
    \mathds{1}(\a_1 \geq -s) \prod_{k=1}^{\infty} I(\a_k)
    }
    \geq 
    \pth{1- 2\kappa\exp \pth{ - \kappa s^{1-2\delta} }}
    \exp \pth{ 
    - \frac{2T^{1/3}s^{5/2}}{15\pi} ( 1 + 9 \ep )
    } \, .
    \label{5.25.anal.eqn}
\end{align}

\begin{proof}[Proof of Claim~\ref{claim.eqn.5.25}]
Negating both sides of \eqref{5.14.anal.eqn} and then exponentiating yields
\[
\prod_{k=1}^{\infty} I(\a_k) \geq \exp \pth { -\frac{1}{2}\mc{L}_{T, \ep}(s+ C_{\ep}^{\GOE}) } \, .
\]
Since $\mc{L}_{T,\ep}(x)$ is monotonically increasing, we may bound
\begin{align}
    \E_{\GOE} \brak{
    \mathds{1}(\a_1 \geq -s) \prod_{k=1}^{\infty} I(\a_k)
    }
    \geq 
    \P \pth{ 
    \a_1 \geq -s, C_{\ep}^{\GOE} < s^{1-\delta}
    }
    \exp \pth{ 
    - \frac{1}{2}\mc{L}_{T, \ep}(s+ s^{1-\delta})
    } \, . \label{5.26.anal.eqn}
\end{align}
Take $S_0 >0$ large enough so that for all $s \geq S_0$,
\begin{align}
    \mc{L}_{T, \ep}(s+ s^{1-\delta}) \leq T^{1/3}  \frac{4s^{5/2}}{15 \pi} (1 + 9 \ep) \, . 
    \label{5.27.anal.eqn}
\end{align}
From Theorem~\ref{1.6_anal}, there exist $\kappa := \kappa (\ep, \delta)$ and a (potentially larger) $S_0$ such that. for all $s \geq S_0$, 
\[
\P (C_{\ep}^{\GOE} <s^{1- \delta}) > 1- \kappa \exp(- \kappa s^{1-2\delta}) \, .
\]
Furthermore, for large enough $S_0$, we find from \eqref{4.4.TWboundeqn} that for all $s \geq S_0$,
\begin{align*}
    \P(\a_1 < - s ) \leq \exp \pth{ 
    -\frac{1}{24}  s^3 (1 + o(1)) 
    }
    \leq \kappa \exp ( - \kappa s^{1-2\delta}) \, .
\end{align*}
Thus, for large enough $S_0$, we have
\begin{align*}
    \P \pth{ 
    \a_1 \geq -s, ~C_{\ep}^{\GOE} < s^{1-\delta}
    } 
    \geq  \P(\a_1 \geq -s) + \P( C_{\ep}^{\GOE} < s^{1-\delta} ) -1 
    \geq 1- 2\kappa\exp \pth{ - \kappa s^{1-2\delta} } \, .
\end{align*}
Plugging this and \eqref{5.27.anal.eqn} into \eqref{5.26.anal.eqn} yields equation \eqref{5.25.anal.eqn} of Claim~\ref{claim.eqn.5.25}.
\end{proof}
\end{claim}

\begin{claim} \label{Claim.eqn.5.28}
There exist constants $K_2 := K_2(T_0) >0$ and $S_0 := S_0( \ep, \delta, T_0) > 0$ such that for all $s \geq S_0$, we have
\begin{align}
    \E_{\GOE} \brak{
    \mathds{1}(\a_1 < -s ) \prod_{k=1}^{\infty} I(\a_k) 
    }
    \geq 
    \exp \pth{ - K_2s^3}
    .
    \label{5.28.anal.eqn}
\end{align}

\begin{proof}[Proof of Claim~\ref{Claim.eqn.5.28}]
Define the parameter $L := \frac{3}{1-\delta}$, and note that $L \in (3, 4]$. Let $\mathfrak{J}$ denote the interval $[-s^L, -s)$. 
We seek an upper bound first on $\sum_{\a_k \in \fk{J}} J_s(\a_k)$ and then on $\sum_{\a_k <-s^L} J_s(\a_k)$. 
Since $J_s(\cdot)$ is monotonically increasing, we obtain the following upper bound by replacing all the $\a_k$'s inside the interval 
$\fk{J}$ by the right endpoint $s$ of the interval:
\begin{align}
    \sum_{\a_k \in \fk{J}} J_s(\a_k) 
    \leq 
    \chi^{\GOE}(\fk{J}) J_s(-s)
    = 
        \frac{1}{2} 
        \chi^{\GOE}(\fk{J}_{\ell}) 
        \log 2
    \, .
\end{align}
Next, using Theorem~\ref{1.5_anal}, there exists $\mc{C}:= \mc{C}(\ep, \delta)$ and $S_0:= S_0(\ep)$ such that for all $s \geq S_0$, we have 
\begin{align}
\chi^{\GOE}(\fk{J}) \leq \E \brak{ \chi^{\GOE} (\fk{J})} + \ep s^{3L/2}
\label{4.5.proof.chi.bd}
\end{align}
holds with probability greater than or equal to $1- \exp(- \mc{C} s^3)$. 
In what follows, we will write $C$ to denote a positive constant independent of $\ep \in (0,1/3)$ and $\delta \in (0,1/4)$ (but may depend on $T_0$) whose value may change from line to line. 
Then from Theorem~\ref{1.3_anal_prop}, we have for large enough $s$ 
\begin{align}
    \E\brak{\chi^{\GOE}(\fk{J})
    } 
    = 
    \frac{2}{3\pi}
    (s^{3L/2}-s^{3/2}) + \fk{D}_1( s^L) - \fk{D}_1(s)
    \leq C s^{3L/2} \, .
\end{align}
Substituting this into \eqref{4.5.proof.chi.bd}, we may deduce that 
\begin{align}
    \sum_{\a_k \in \fk{J}} J(\a_k) 
    \leq 
    Cs^{3L/2} 
    \label{5.29.anal.eqn}
\end{align}
holds with probability greater than or equal to $1-  \exp(- \mc{C} s^3)$.

It remains to bound the sum $\sum_{\a_k < - s^L} J_s(\a_k)$, which we now decompose into two sums:
\begin{align}
    \sum_{\a_k < - s^L} J_s(\a_k)&
    = (\rm{\bf{A}}) + (\rm{\bf{B}}), ~\text{where} \\
    (\rm{\bf{A}}) 
    := \sum_{\{k: ~\a_k < - s^L, ~\ld_k \leq  s^L\}} J_s(\a_k)&,
    \ \ \ \ \ 
    (\rm{\bf{B}}) := 
    \sum_{\{k: ~\a_k < - s^L, ~\ld_k >  s^L\}} J_s(\a_k) \, .
\end{align}
Using  the bound $\log(1+a) \leq a$ for all $a\geq 0$ gives
\[
J_s(\a_k)  \leq \frac{1}{2}\exp \pth{ 
T^{1/3} \pth{ s-  s^L}
}
\leq \frac{1}{2}\exp \pth{ -(1- \ep)T^{1/3} s^3 } \, ,
\]
for $\a_k \leq -  s^L$, $S_0 := S_0(\ep, \delta)$ large enough, and all $s \geq S_0$.
Corollary \ref{1.5_anal_expchi} shows  
\[
\#\{k : \ld_k \leq  s^L \} 
= 
\frac{2}{3 \pi} s^{3L/2} + C_1(s^L)
\leq Cs^{ 3L/2 } \, .
\]
Thus, for large enough $S_0$,
we have 
\begin{align}
    (\rm{\bf{A}}) \leq 
    \frac{1}{2}Cs^{ 3L/2 } 
    \exp \pth{ 
    -(1- \ep)T^{1/3} s^3)
    } \leq s^3.
    \label{eqn:A-bound}
\end{align}

We now bound $(\rm{\bf{B}})$. 
From monotonicity and \eqref{1.6.anal.eqn}, we have 
$J_s(\a_k) \leq J_s\pth{ 
-(1-\ep) \ld_k + C_{\ep}^{\GOE}
}$, where $C_{\ep}^{\GOE}$ is as defined in Theorem~\ref{1.6_anal}.  
We now employ Theorem~\ref{1.6_anal}, taking $\tilde{s}$ and $\tilde{\delta}$ as our variables instead of the $s$ and $\delta$ in the notation of the theorem to avoid confusion (though we take the $\ep$ in the statement of Theorem~\ref{1.6_anal} to be the same as our $\ep$ here).
With
 $\tilde{s} := s^{3 + \frac{\delta}{2}}$ and $\tilde{\delta} := \frac{\delta}{2(3+ \delta/2)}$,
Theorem~\ref{1.6_anal} implies
that there exist $\kappa := \kappa(\ep, \delta) > 0$ and $S_0 := S_0 (\ep, \delta) > 0$ such that 
for all $s \geq S_0$, we have
\[
\P\pth{C_{\ep}^{\GOE} < s^{3 + \frac{\delta}{2}} }
\geq 
1- \kappa \exp \pth{-\kappa s^3} \, .
\]
Now, for large enough $S_0$, we have $s+ s^{3+ \frac{\delta}{2}} \leq (1-\ep)  s^L$. Since $s^L < \ld_k$ in $(\rm{\bf{B}})$, we have for large enough $S_0$
\begin{align}
    \P \pth{ 
    (\rm{\bf{B}}) \leq \sum_{\ld_k >  s^L} 
    J_s \pth{ 
        (1-\ep)(s^L - \ld_k ) -s
        }
    } 
    \geq 1- \kappa \exp \pth{ 
    - \kappa s^3
    }.
    \label{5.31.eqn.anal}
\end{align}
The bounds in \eqref{eqn:A-bound}, \eqref{5.31.eqn.anal}, and \eqref{5.35.eqn.anal} of Claim~\ref{lemma.5.5.anal} (given below), as well as the bound $3L/4 \leq 3$, we find that for $S_0$ large enough, 
\begin{align}
    \P \pth{
    (\rm{\bf{A}}) + (\rm{\bf{B}}) 
    \leq 
    Cs^{3}
    } \geq 1- \kappa \exp \pth{ - \kappa s^3}
    \label{eqn:A+B-delta}
\end{align}
Combining this bound with the bound  in \eqref{5.29.anal.eqn} yields 
\begin{align}
    \P( \mc{A} ) \geq 
    1
    -  \exp( - \mc{C} s^3)
    - \kappa \exp \pth{ - \kappa s^3},
    \label{5.32.anal.eqn}
\end{align}
where $\mc{A} := \left \{ 
\sum_{k=1}^{\infty} J_s(\a_k) \leq Cs^{3}
\right \}$.
We then obtain
\begin{align}
    \E_{\GOE} \brak{
    \mathds{1}(\a_1 < -s ) \prod_{k=1}^{\infty} I(\a_k) 
    }
    \geq 
    \P \pth{\{\a_1 < -s\} \cap \mc{A}} \exp(- C s^{3}) \, .
    \label{5.33.anal.eqn}
\end{align}
We finally estimate, for a constant $K_2 >0$ and for large enough $S_0$,
\begin{align}
    \P \pth{\{\a_1 \leq -s\} \cap \mc{A}} 
    &\geq 
    \P(\a_1 \leq -s ) + \P(\mc{A}) - 1 \nonumber \\
    &\geq 
    \exp\pth{-s^3}
    -  \exp( - \mc{C} s^3)
    - \kappa \exp \pth{ - \kappa s^3} \nonumber \\
    &\geq \exp\pth{-C' s^3},
    \label{5.34.anal.eqn}
\end{align}
where the first inequality uses 
$\P( A \cap B) \geq \P(A) + \P(B) -1$ for any events $A$ and $B$, and the second inequality uses  \eqref{4.4.TWboundeqn} and the lower bound in \eqref{5.32.anal.eqn}.
Substituting \eqref{5.34.anal.eqn} into \eqref{5.33.anal.eqn} yields  \eqref{5.28.anal.eqn}.
This concludes the proof of Claim~\ref{Claim.eqn.5.28}.
\end{proof}
\end{claim}

We may now complete the proof of \eqref{4.5.anal.eqn} of Proposition \ref{prop.4.2.anal} by substituting \eqref{5.25.anal.eqn} and \eqref{5.28.anal.eqn} into 
\begin{align}
    \E_{\GOE} \brak{ \prod_{k=1}^{\infty} I(\a_k) 
    } 
    = 
    \E_{\GOE} \brak{
    \mathds{1}(\a_1 \geq -s ) \prod_{k=1}^{\infty} I(\a_k) 
    }
    +
    \E_{\GOE} \brak{
    \mathds{1}(\a_1 < -s ) \prod_{k=1}^{\infty} I(\a_k) 
    }.
\end{align}
\end{proof}

\begin{claim} \label{lemma.5.5.anal}
Fix $\ep \in (0,1/3)$, $\delta \in (0, 1/4)$ and $T_0 > 0$. There exists a positive constant $S_0 := S_0 (\ep, \delta)$ such that for all $s \geq S_0$, we have
\begin{align}
    \sum_{\ld_k >  s^L} J_s \pth{ (1-\ep)( s^L - \ld_k) - s
    }
    &\leq
    Cs^{3L/4} \, .
    \label{5.35.eqn.anal}
\end{align}

\begin{proof}
For sufficiently large $s$, \eqref{1.5_anal_evalest} implies that 
\begin{align}
    \left \{ k : \ld_k >  s^L \right \} \subseteq 
    \left \{ k: k > \frac{2}{3\pi} \pth{ s^L}^{3/2} - \frac{3}{4} \right \}.
    \label{lemma.5.5.k.bound}
\end{align}
This gives
\begin{align}
    \sum_{\ld_k >  s^L} J_s \pth{ (1-\ep)( s^L - \ld_k) - s
    }
    &\leq 
    \sum_{ k > \frac{2}{3\pi}s^{3L/2} - \frac{3}{4}}
    J_s\pth{ 
    (1- \ep)(s^L - \ld_k) -s
    }. \label{lemma.5.5.eqn.1}
\end{align}
To simplify the calculations that follow, we denote
$\theta_0 := \frac{2}{3\pi}s^{3L/2}- \frac{3}{4}$ and 
$\theta_0' := \theta_0+ \sqrt{\frac{2}{\pi}}s^{3L/4}$. 
Note that for $\ld_k > \theta_0$, we have
$(1-\ep)( s^L - \ld_k) - s < 0$ for sufficiently large $S_0$. We then use the fact that, for $x\leq -s$, we have 
$J_s(x) \leq \frac{1}{2} \log 2$.  
This is the bound we take on $J_s(\cdot)$ for $k  \in [\theta_0 , \theta_0']$. 

For $ k > \theta_0'$, we recall the inequality $\log(1+z) \leq z$ for $z \geq 0$, which gives 
\begin{align}
    J_s((1-\ep)( s^L - \ld_k) - s) 
    &\leq 
    \frac{1}{2}\exp \pth{(1-\ep) T^{1/3} (s^L - \ld_k)} \, .
    \label{lemma.5.5.eqn.2}
\end{align}
Define $\bar{k} := k - \frac{1}{4} + \mc{R}(n)$ and $k' := k - \theta_0$, and note that $\bar{k} > \theta_0$ for $k > \theta_0'$. Then Taylor's theorem yields
\begin{align}
    s^L - \ld_k  
    = 
    \pth{\frac{3\pi}{2} \Big (\theta_0+\frac{3}{4} \Big )}^{2/3} 
    - \pth{
    \frac{3\pi}{2} \bar{k}
    }^{2/3} 
    \leq - C (k')^{2/3} \, .
    \label{lemma.5.5.eqn.3}
\end{align}
Now, substituting the bound given in \eqref{lemma.5.5.eqn.3}
into \eqref{lemma.5.5.eqn.2} yields
\begin{align}
    J_s((1-\ep)( s^L - \ld_k) - s) 
    &\leq 
    \begin{cases}
        \frac{1}{2}\log 2 &k \in [\theta_0, \theta_0'] \cap \Z \\
        \frac{1}{2} \exp \pth{ 
        - C (1-\ep)T^{1/3} (k')^{2/3}
        } 
        &k \in (\theta_0', \infty) \cap \Z
    \end{cases} 
    \, .
\end{align}
From this bound, we have
\begin{align}
    \sum_{\ld_k >  s^L} J_s \pth{ (1-\ep)( s^L - \ld_k) - s
    } &\leq 
    \frac{1}{2} (\theta_0' - \theta_0) \log 2 
    + 
    \frac{1}{2}\sum_{k' > \theta_0' - \theta_0} 
    \exp \pth{ 
        - C (1-\ep)T^{1/3} (k')^{2/3}
        } \\
    &\leq \frac{1}{\sqrt{2\pi}} s^{3L/4} \log 2
    + \frac{C}{(1-\ep) T^{1/3}} \\
    &\leq C s^{3L/4} \, ,
\end{align}
where the second-to-last inequality follows by bounding the sum with an integral. This gives the claim.
\end{proof}
\end{claim}

\newpage
\bibliography{references}{}
\nocite{*}
\bibliographystyle{alpha}

\end{document}